\documentclass
[a4paper,
11pt,
twoside
]{article}

\usepackage[T1]{fontenc}
\usepackage{ae}

\usepackage[left=1in, right=1in, top=1.5in, bottom=1.0in, headheight=15pt]{geometry}
\usepackage{authblk} 

\usepackage{amsmath}
\usepackage{amssymb}
\usepackage{amsthm}

\usepackage{charter,courier}  
\usepackage{chemarrow}
\usepackage{caption}   
\usepackage[figuresright]{rotating} 

\usepackage{fancyhdr}
\usepackage{empheq}

\usepackage{latexsym}
\usepackage{lscape}

\usepackage{mleftright} 

\usepackage{makeidx}
\makeindex
\usepackage[nottoc]{tocbibind}

\usepackage[intoc]{nomencl}
\makenomenclature
\usepackage{stmaryrd}
\usepackage{textcomp}
\usepackage{tikz}

\usepackage{wasysym}

\usepackage{xcolor}
\usepackage{color}
\usepackage{colortbl}

\usepackage{yhmath}
\usepackage{graphicx}

\usepackage[all]{xy}

\definecolor{colComments}{rgb}{1,0,0}

\theoremstyle{plain}  

\swapnumbers 
\numberwithin{equation}{section} 
\newtheorem{Definition}[equation]{Definition} 
\newtheorem{Definition/Lemma}[equation]{Definition/Lemma}

\newtheorem{Lemma}[equation]{Lemma}
\newtheorem{Theorem}[equation]{Theorem}
\newtheorem{Proposition}[equation]{Proposition}
\newtheorem{Corollary}[equation]{Corollary}
\newtheorem{Remark}[equation]{Remark}

\newtheorem{Notation}[equation]{Notation}
\newtheorem{Notation/Lemma}[equation]{Notation/Lemma}

\newtheorem{Comparison}[equation]{Comparison}

\font\triangles=beta
\newcommand{\Squar}{\mbox{\triangles 3}}
\newcommand{\UR}{\mbox{\triangles 1}}
\newcommand{\UP}{\mbox{\triangles 2}}

\font\trianglesb=betab
\newcommand{\Squarb}{\mbox{\trianglesb 3}}
\newcommand{\URb}{\mbox{\trianglesb 1}}
\newcommand{\UPb}{\mbox{\trianglesb 2}}

\title{A supercharacter theory for Sylow $p$-subgroups\\
         of the Steinberg triality groups
         }
\author{Yujiao Sun}
\affil{\footnotesize School of Mathematics and Statistics,
         Beijing Institute of Technology,\\
      Beijing 100081, PR China}
\affil{\footnotesize E-mail:
           yujiaosun@bit.edu.cn}

\date{}

\begin{document}
\maketitle

\begin{abstract}
We determine a supercharacter theory for the matrix Sylow $p$-subgroup ${^3}D_4^{syl}(q^3)$
of the Steinberg triality group ${^3}D_4(q^3)$, and establish the supercharacter table
of ${^3}D_4^{syl}(q^3)$.
\end{abstract}




\section{Introduction}

Let $p$ be a fixed prime number,
$\mathbb{N}^*$ the set $\{1,2,3,\dots\}$ of positive integers,
$q:=p^k$ with a fixed $k\in \mathbb{N}^*$,
$\mathbb{F}_q$ the finite field with $q$ elements,
and $A_n(q)$ ($n\in \mathbb{N}^*$) the group of upper unitriangular $n\times n$-matrices
over $\mathbb{F}_q$.
Thus $A_n(q)$ is a Sylow $p$-subgroup
of the general linear group $GL_n(q)$
and also a Sylow $p$-subgroup \cite{Carter1}
of the Chevalley group of Lie type $A_{n-1}$ over $\mathbb{F}_q$ if $n \geq 2$.
Classifying the conjugacy classes of $A_n(q)$ and hence the complex irreducible characters
is known to be a ``wild'' problem, see e.g. \cite{Evs2011, ps, vla}.


The notion of {\it supercharacter theory} (see \ref{supercharacter theory}) for an arbitrary finite group
was introduced by P. Diaconis and I.M. Isaacs \cite{di},
which is a coarser approximation of the character theory.
\begin{itemize}
\item A supercharacter theory replaces irreducible characters by {\it supercharacters}.
Supercharacters are pairwise orthogonal.
Each irreducible character is a constituent of precisely one supercharacter.
\item A supercharacter theory replaces conjugacy classes by {\it superclasses}.
    The number of different supercharacters equals the number of superclasses.
Supercharacters are constant on superclasses.
\item A supercharacter theory replaces irreducible modules by {\it supermodules}.
\item A {\it supercharacter table} is constructed as a replacement for the character table.
\end{itemize}
C.A.M. Andr\'{e} \cite{and1} using the Kirillov orbit method,
and later but independently N. Yan \cite{yan2} using a more elementary method
determined a supercharacter theory for $A_n(q)$,
the {\it Andr\'{e}-Yan supercharacter theory},
which has beautiful combinatorial properties.
P. Diaconis and I.M. Isaacs \cite{di} extended
this theory to so-called algebra groups.
The supercharacter theory for $A_n(q)$ is based on
the observation that $u\mapsto u-1$ defines a bijection from $A_n(q)$
to an $\mathbb{F}_q$-vector space of nilpotent upper triangular matrices
which is the Lie algebra associated with $A_n(q)$.
Unfortunately, this does not work in general for Sylow $p$-subgroups
of the other finite groups of Lie type.

C.A.M. Andr\'{e} and A.M. Neto studied in \cite{an1,an2, an3} supercharacter theories
for Sylow $p$-subgroups of untwisted Chevalley groups of types $B_n$, $C_n$ and $D_n$
which are finite classical groups of untwisted Lie type.
C.A.M. Andr\'{e}, P.J. Freitas and A.M. Neto \cite{AFM2015}
extended in a uniform way the construction of \cite{an1, an3} to
Sylow $p$-subgroups of
finite classical groups of untwisted Lie type.
This has been extended by S. Andrews \cite{Andrews2015, Andrews2016}
to unitary groups as well.
Supercharacters of those groups arise as restrictions of supercharacters of overlying
full upper unitriangular groups $A_N(q)$ to the Sylow $p$-subgroups of classical type,
and superclasses arise as intersections of superclasses of $A_N(q)$ with these groups.

As mentioned above the Andr\'{e}-Yan supercharacter theory for $A_n(q)$
relies on the map $u\mapsto u-1$ from $A_n(q)$ onto the Lie algebra.
M. Jedlitschky  generalised this by a procedure called
{\it monomial linearisation}  (see \cite[\S 2.1]{Markus1}) for a finite group,
and decomposed {Andr\'{e}-Neto supercharacters} for Sylow $p$-subgroups
of Lie type $D$ into much smaller characters \cite{Markus1}.
The smaller characters are determined by certain monomial transitive representations
of the group,
which are pairwise orthogonal
and have the property that each irreducible character is a constituent of exactly one of those.
Thus these characters look like supercharacters for a much finer supercharacter theory
for the Sylow $p$-subgroups of Lie type $D$.
One may ask, if there exist corresponding superclasses producing
a full supercharacter theory
for the group which is much finer than the one defined by Andr\'{e} and Neto.
So far there are no corresponding finer superclasses for the Sylow $p$-subgroups
of type $D$.
On the other hand, the Andr\'{e}-Yan supercharacters for $A_n(q)$ are special cases
of Jedlitschky's construction.
Thus one might ask if there is a supercharacter theory for
Sylow $p$-subgroups of all finite groups of Lie type
behind this specialising to the Andr\'{e}-Yan theory in type $A$
and to the characters defined by Jedlitschky in type $D$.

Finite groups of exceptional Lie type are series of matrix groups of fixed size,
and only the prime power $q$
of the underlying field is varied.
Hence it seems to be reasonable to try the exceptional types first,
apply Jedlitschky's monomial linearisation  to obtain supercharacters and then supplement it
to construct superclasses as well in order to exhibit a full supercharacter theory.
This will be done in this paper in the special case of the twisted Lie type $^3D_4$:
the Sylow $p$-subgroup ${{^3D}_4^{syl}}(q^3)$
of the Steinberg triality group ${{D}_4^{syl}}(q^3)$.
For this group, irreducible characters have been classified by T. Le \cite{Le}
and their character tables have been given by the author explicitly in \cite{sun2016arxiv}.
In particular, as opposed to the case $A_n(q)$,
the problem of determining the irreducible characers for all $q$ is not ``wild''.
Thus for ${{^3D}_4^{syl}}(q^3)$ we can in addition explicitly
decompose the supercharacters into irreducible characters.
The method of this paper seems to work for more exceptional Lie types,
indeed in forthcoming papers we shall obtain similar results
for the cases of type $G_2$ and twisted type $^2G_2$.
Thus we have some evidence that there is indeed a general supercharacter theory
for all Lie types behind this.

For the matrix Sylow $p$-subgroup ${{^3D}_4^{syl}}(q^3)$ (see Section \ref{sec:3D4-1-3})
of the Steinberg triality group we shall proceed as follows:
\begin{itemize}
 \item [1.] {\it Determine a monomial module by constructing a monomial linearisation:}
            Determine an intermediate algebra group $G_8(q^3)$ which is between ${{^3D}_4^{syl}}(q^3)$
            and the overlying full upper unitriangular group $A_8(q^3)$ (see Section \ref{sec:bigger group}),
            then construct a monomial linearisation for $G_8(q^3)$
            and obtain a monomial $G_8(q^3)$-module $\mathbb{C}\left({{^3D}_4^{syl}}(q^3) \right)$
            (see Section \ref{sec: monomial 3D4-module}).
 \item [2.] {\it Establish supercharacters of ${{^3D}_4^{syl}}(q^3)$ by decomposing monomial ${{^3D}_4^{syl}}(q^3)$-modules:}
            Every supercharacter is afforded by a direct sum of
            some ${{^3D}_4^{syl}}(q^3)$-orbit modules
            which is also a direct sum of restrictions of certain $G_8(q^3)$-orbit modules
            to ${{^3D}_4^{syl}}(q^3)$
            (see Sections \ref{sec:U-orbit modules-3D4}, \ref{sec:homom. between orbit modules-3D4}
             and \ref{sec: supercharacter theories-3D4}).
 \item [3.] {\it Calculate the superclasses using the intermediate group $G_8(q^3)$:}
            Every superclasses is a union of some intersections of biorbits of $G_8(q^3)$ and ${{^3D}_4^{syl}}(q^3)$,
            i.e. $\{I_8+g(u-I_8)h\mid g, h \in G_8(q^3)\}\cap {{^3D}_4^{syl}}(q^3)$ for all $u\in {{^3D}_4^{syl}}(q^3)$,
            where $I_8$ is the identity element of ${{^3D}_4^{syl}}(q^3)$
            (see Sections \ref{partition of U-3D4} and \ref{sec: supercharacter theories-3D4}).
\end{itemize}
At the end of every section,
we compare the properties of $A_n(q)$, $D_n^{syl}(q)$ and ${^3}D_4^{syl}(q^3)$.

We mention that supercharacter theories have proven to raise other questions in particular concerning algebraic combinatorics.
For instance, A.O.F. Hendrickson in \cite{Hend}
discovered the relation between supercharacter theories and Schur rings.

We fix some notation:
Let
$p$ be a fixed odd prime,
$q$ a fixed power of $p$,
$K$ a field,
$K^*$ the multiplicative group $K\backslash\{0\}$ of $K$,
$K^+$ the additive group of $K$,
$\mathbb{F}_q$ the finite field with $q$ elements,
$\mathbb{F}_{q^3}$ the finite field with $q^3$ elements,
$\mathbb{C}$ the complex field,
$\mathbb{Z}$ the set of all integers,
$\mathbb{N}$ the set $\{0,1,2, \dots \}$ of all non-negative integers,
$\mathbb{N}^*$ the set $\{1,2,\dots \}$ of all positive integers.
Let $\mathrm{Mat}_{8 \times 8}(K)$
be the set of all $8\times 8$ matrices with entries in the field $K$,
the {general linear group}
${GL}_8(K)$ be the subset of $\mathrm{Mat}_{8\times 8}(K)$
consisting of all invertible matrices.
For $m\in \mathrm{Mat}_{8\times 8}(K)$, let  $m:=(m_{i,j})$,
where $m_{i,j}\in K$ denotes the $(i,j)$-entry of $m$ (the entry in the $i$-th row and $j$-th column).
For simplicity, we write $m_{ij}:=m_{i,j}$ if there is no ambiguity.
Denote by $e_{i,j}\in \mathrm{Mat}_{8\times 8}(K)$ the matrix unit
with $1$ in the $(i,j)$-position and $0$ elsewhere.
Denote by $A^{\top}$ the {transpose} of $A\in \mathrm{Mat}_{8\times 8}(K)$.
Let $O_8$ be the zero $8 \times 8$-matrix $O_{8\times 8}$,
and $1$ denote the identity element of a finite group.

\section{Sylow $p$-subgroup ${^3}D_4^{syl}(q^3)$ of the Steinberg triality group}
\label{sec:3D4-1-3}
In this section,
we exhibit a matrix Sylow $p$-subgroup $D_4^{syl}(q)$ of the Chevalley group $D_4(q)$ of type $D_4$.
The Steinberg triality group ${^3}D_4(q^3)$
is generated by the fixed elements
under an automorphism of $D_4(q^3)$  of order 3
which arises from a field automorphism of $D_4(q^3)$ and a graph automorphism of $D_4(q^3)$.
The Sylow $p$-subgroup ${^3}D_4^{syl}(q^3)$
of ${^3}D_4(q^3)$
is generated by the fixed elements of $D_4^{syl}(q^3)$ under the above mentioned automorphism.
The main references are
\cite{Carter1, Carter2005}.

Let
${J_{8}^+}:=
\sum_{i=1}^8{e_{i,9-i}}\in {GL}_{8}(\mathbb{C})$.
Then $
\{A \in \mathrm{Mat}_{8\times 8}(\mathbb{C}) \mid  A^\top J_{8}^+ + J_{8}^+ A=0 \}
$
forms the complex simple Lie algebra of type $D_4$.
For $1\leq i\leq 4$, let
$h_i:= e_{i,i}-e_{9-i,9-i} \in \mathrm{Mat}_{8\times 8}(\mathbb{C})$.
Then a Cartan subalgebra of $\mathcal{L}_{D_4}$ is
$\mathcal{H}_{D_4}= \{\sum_{i=1}^4{\lambda_ih_i}  \mid \lambda_i\in \mathbb{C}\}$.
Let  $\mathcal{H}_{D_4}^*$ be the dual space of $ \mathcal{H}_{D_4}$,
and $h:=\sum_{i=1}^4{\lambda_ih_i}$.
For $1 \leq i \leq 4$,
let $\varepsilon_i \in \mathcal{H}_{D_4}^*$
be defined by $\varepsilon_i(h)=\lambda_i$ for all $i=1,2,3,4$.
Let $\mathcal{V}_4:=\mathcal{V}_{D_4}$
be a $\mathbb{R}$-vector subspace of $\mathcal{H}_{D_4}^*$
spanned by $\{ h_i \mid i=1,2,3,4\}$.
Then $\mathcal{V}_4$ becomes a Euclidean space (see \cite[\S 5.1]{Carter2005}).
The set
$\Phi_{D_4}=\{\pm\varepsilon_i\pm\varepsilon_j  \mid  1\leq i< j\leq 4\}$
is a root system of type ${D}_4$.
The fundamental system of roots of the root system $\Phi_{D_4}$ is
$\Delta_{D_4}
=\{\varepsilon_1-\varepsilon_2,\ \varepsilon_2-\varepsilon_3,
\ \varepsilon_3-\varepsilon_4,\ \varepsilon_3+\varepsilon_4\}$.
The {positive system of roots} of $\Phi_{D_4}$ is
$\Phi_{D_4}^+:=\{\varepsilon_i\pm\varepsilon_j  \mid  1 \leq i< j\leq4\}$.
We choose the following Chevalley basis of $\mathcal{L}_{D_4}$
to keep the description of the Steinberg triality group
as simple as possible.
\begin{Lemma}[Chevalley basis of $\mathcal{L}_{D_4}$]\label{Chevalley, D4}
Let
$ r_1:=\varepsilon_1-\varepsilon_2$,
$r_2:=\varepsilon_2-\varepsilon_3$,
$r_3:=\varepsilon_3-\varepsilon_4$,
$ r_4:=\varepsilon_3+\varepsilon_4$,
$ r_5:= r_1+r_2$,
$ r_6:= r_2+r_3$,
$ r_7:= r_2+r_4$,
$ r_8:= r_1+r_2+r_3$,
$ r_{9}:= r_1+r_2+r_4$,
$ r_{10}:= r_2+r_3+r_4$
$ r_{11}:= r_1+r_2+r_3+r_4$,
and
$ r_{12}:= r_1+2r_2+r_3+r_4$.
Then the Lie algebra $\mathcal{L}_{D_4}$ has
a Chevalley basis $\{h_r \mid r\in \Delta_{D_4}\}\cup \{ e_{\pm r} \mid r\in \Phi_{D_4}^+\}$,
where
\begin{alignat*}{4}
& e_{r_1}:=e_{1,2}-e_{7,8},
& \quad
& e_{r_2}:=e_{2,3}-e_{6,7},
& \quad
& e_{r_3}:=e_{3,4}-e_{5,6},
& \quad
& e_{r_4}:=e_{3,5}-e_{4,6},\\
& e_{r_5}:={-}(e_{1,3}-e_{6,8}),
& \quad
& e_{r_6}:=e_{2,4}-e_{5,7},
& \quad
& e_{r_7}:=e_{2,5}-e_{4,7},
& \quad
& e_{r_8}:=e_{1,4}-e_{5,8},\\
& e_{r_{9}}:=e_{1,5}-e_{4,8},
& \quad
& e_{r_{10}}:=e_{2,6}-e_{3,7},
& \quad
& e_{r_{11}}:=e_{1,6}-e_{3,8},
& \quad
& e_{r_{12}}:=e_{1,7}-e_{2,8},
\end{alignat*}
and $e_{-r}:=e_r^\top$ for all $r\in \Phi_{D_4}^+$,
$h_r:=[e_r,e_{-r}]=e_re_{-r}-e_{-r}e_r$ for all  $r\in \Delta_{D_4}$.
\end{Lemma}

Set the matrix group
$\bar{D}_4(q):=\left<{\,} \exp(t{\,e_r}) {\,}\middle|{\,}  r\in \Phi_{D_4},{\ } t\in \mathbb{F}_q {\,}\right>$.
The Chevalley group ${D}_4(q)$ of type $\mathcal{L}_{D_4}$ over the field $\mathbb{F}_q$
is denoted by
${D}_4(q):=\left<{\,} \exp(t{\,}\mathrm{ad}{\,e_r}) {\,}\middle|{\,}  r\in \Phi_{D_4},{\ } t\in \mathbb{F}_q {\,}\right>$.
By the adaption of \cite[11.3.2]{Carter1},
$\bar{D}_4(q)$ is the commutator subgroup, denoted by $\Omega_{8}(q,f_D)$,
of the (positive) orthogonal group $O_{8}(q,f_D)$,
where $f_D$ is the quadratic form $x_1x_{8}+x_2x_{7}+x_3x_{7}+x_4x_{5}$.
By \cite[\S 11.3]{Carter1},
$D_4(q) \cong \bar{D}_4(q)/Z(\bar{D}_4(q))$,
and $D_4(q)$ is isomorphic to
the projective group
$P\Omega_{8}(q,f_D)=\Omega_{8}(q,f_D)/{Z(O_{8}(q,f_D))\cap \Omega_{8}(q,f_D)}$.
Set
$D_4^{syl}(q):=\left<{\,}\exp(te_r) {\,}\middle|{\,} r\in \Phi_{D_4}^+,{\ } t\in \mathbb{F}_q {\,}\right>$.
Since $D_4^{syl}(q)\cap Z(O_{8}(q,f_D))=\{I_8\}$,
$D_4^{syl}(q)$ is a Sylow $p$-subgroup of the Chevalley group ${D}_4(q)$.
Since $|{D}_4(q)|= \frac{1}{4}q^{12}(q^2-1)(q^4-1)(q^6-1)(q^4-1)$ (see \cite{Carter1965}),
$|D_4^{syl}(q)|=q^{12}$.
Comparing the orders of $O_{8}(q,f_D)$ and $P\Omega_{8}(q,f_D)$ (see \cite[Chapter 9]{Grove}),
$D_4^{syl}(q)$ is also a Sylow $p$-subgroup of $O_{8}(q,f_D)$.
We set
${x_r(t)}:= \exp(te_r)=I_{8}+t\cdot e_r$
for all $r\in \Phi_{D_4}$ and $t\in \mathbb{F}_q$,
and the root subgroups
${X_r}:= \{ {x_r(t)} \mid t\in {\mathbb{F}_q}\}$
for all $r\in \Phi_{D_4}$.

 Define the sets of matrix entry coordinates:
  $\Squar:= \{(i,j) \mid  1\leq i,j \leq 8 \}$,
  $\UR:= \{(i,j) \mid 1\leq i< j \leq 8 \}$
  and
  $\UP:= \{(i,j)\in \Squar  \mid  i < j < 9-i \}$.
The following result is well known.
\begin{Lemma} \label{pairwise diff-3D4}
For $(i,j)\in \UR$ and $t \in \mathbb{F}_q$,
set $\tilde{x}_{i,j}(t):=I_{n}+te_{i,j}$.
For $n\in \mathbb{N}^*$, we have
$A_{n}(q)=\Big\{\prod_{(i,j)\in \URb}{\tilde{x}_{i,j}(t_{i,j})}
 {\,}\Big|{\,}
 t_{i,j} \in \mathbb{F}_q \Big\}$,
where the product can be taken in an arbitrary, but fixed, order.
In particular,
$
\prod_{(i,j)\in \URb}{\tilde{x}_{i,j}(t_{i,j})}
=\prod_{(i,j)\in \URb}{\tilde{x}_{i,j}(s_{i,j})}
 \Longleftrightarrow
t_{i,j}=s_{i,j}
$
 for all $(i,j)\in \UR$.
\end{Lemma}
A good expression for each element of $D_4^{syl}(q)$ is obtained from the following result.
\begin{Lemma}[c.f. Lemma 17 of \cite{steinberg1968}]\label{sylow p-subg, D4}
$
 D_4^{syl}(q)=\{\prod_{r\in \Phi^+_{D_4}}{x_r(t_r)}
 {\,}|{\,}
 t_r \in \mathbb{F}_q \}$,
where the product can be taken in an arbitrary, but fixed, order.
In particular,
\begin{align*}
 &D_4^{syl}(q)=
 \left\{
 \begin{array}{l}
 x_{r_4}(t_{r_4})x_{r_3}(t_{r_3}) x_{r_{10}}(t_{r_{10}})x_{r_7}(t_{r_7})
  x_{r_6}(t_{r_6})x_{r_2}(t_{r_2}) \\
  \cdot x_{r_{12}}(t_{r_{12}})
   x_{r_{11}}(t_{r_{11}})x_{r_{9}}(t_{r_{9}})x_{r_{8}}(t_{r_{8}})
   x_{r_{5}}(t_{r_{5}})x_{r_{1}}(t_{r_{1}})
 \end{array}
   {\,}\middle|{\,}
  t_{r_i}\in \mathbb{F}_q, {\ } i=1,2,\dots, 12\right\}.
 \end{align*}
\end{Lemma}


Now we determine the Sylow $p$-subgroup ${^3}D_4^{syl}(q^3)$ of the Steinberg triality group.
Let $\rho$ be a linear transformation of $\mathcal{V}_4$ into itself
arising from a non-trivial symmetry of the Dynkin diagram of $\mathcal{L}_{D_4}$
sending $r_1$ to $r_3$, $r_3$ to $r_4$, $r_4$ to $r_1$, and fixing $r_2$.
Then
$r_1 \stackrel{\rho}{\mapsto } r_3 \stackrel{\rho}{\mapsto } r_4$,
$r_2 \stackrel{\rho}{\mapsto } r_2$,
$r_{5} \stackrel{\rho}{\mapsto } r_{6} \stackrel{\rho}{\mapsto } r_{7}$,
$r_{8} \stackrel{\rho}{\mapsto } r_{10} \stackrel{\rho}{\mapsto } r_{9}$,
$r_{11} \stackrel{\rho}{\mapsto } r_{11}$,
$r_{12} \stackrel{\rho}{\mapsto } r_{12}$,
and $\rho^3=\mathrm{id}_{\mathcal{V}_4}$.
Thus $\rho(\Phi_{D_4})=\Phi_{D_4}$ (see \cite[12.2.2]{Carter1}).
Let an automorphism of the Lie algebra $\mathcal{L}_{D_4}$ be determined by:
$h_r\mapsto h_{\rho(r)}$,
$e_r\mapsto e_{\rho(r)}$,
$e_{-r}\mapsto e_{-\rho(r)}$
for all $r\in \Delta_{D_4}$,
and for every $r\in \Phi_{D_4}$ satisfying $e_r\mapsto \gamma_re_{\rho(r)}$.
For the Chevalley basis in \ref{Chevalley, D4},
we have
$\gamma_r=1$ for all $r\in \Phi_{D_4}$.

The Chevalley group $D_4(q^3)$ has a field automorphism $F_q$
sending $x_{r}(t)$ to $x_{r}(t^q)$,
and a graph automorphism $\rho$ sending $x_{r}(t)$ to $x_{\rho(r)}(t)$  $(r\in \Phi_{D_4})$
(see \cite[12.2.3]{Carter1}).
Let $F:=\rho {F_q}={F_q} \rho$.
For a subgroup $X$ of $D_4(q^3)$,
we set $X^F:=\{x\in X| F(x)=x\}$.
Then $D_4(q^3)^F={^3}D_4(q^3)$.
Let $r\in \Phi_{D_4}^+$ and $t\in \mathbb{F}_{q^3}$, set
\begin{align*}
x_{r^1}(t):=
 {\left\{
 \begin{array}{ll}
 x_r(t)& \text{if } \rho(r)=r,{\ }t^q=t\\
 x_r(t)\cdot x_{\rho(r)}(t^q)\cdot x_{\rho^2(r)}(t^{q^2}) & \text{if } \rho(r)\neq r,{\ }t^{q^3}=t\\
 \end{array}
 \right.}.
\end{align*}
Then a Sylow $p$-subgroup of $^3D_4(q^3)$ is determined
\begin{align*}
{}^3{D}_4^{syl}(q^3):
&=\left\{x_{r_2^1}(t_2)x_{r^1_1}(t_1)x_{r^1_5}(t_5)x_{r^1_8}(t_8)x_{r^1_{11}}(t_{11})x_{r^1_{12}}(t_{12})
{\,}\middle|{\,}
{\left\{\begin{array}{l}
t_1, t_5, t_8 \in {\mathbb{F}_{q^3}}\\
t_2, t_{11}, t_{12} \in {\mathbb{F}_{q}}
\end{array}
\right.}
\right\}.
\end{align*}
In particular, $|{}^3{D}_4^{syl}(q^3)|=q^{12}$,
since $|^3{D}_4(q^3)|= q^{12}(q^2-1)(q^6-1)(q^8+q^4+1)$
(see \cite[14.3.2]{Carter1}).

\begin{Definition/Lemma}[in ${}^3{D}_4^{syl}(q^3)$]
\label{root subgroups-3D4}
For $t\in \mathbb{F}_{q^3}$, set
\begin{alignat*}{2}
x_1(t):&=x_{r_1^1}(t)=x_{r_3^1}(t^q)=x_{r_4^1}(t^{q^2})
   =x_{r_1}(t)
  \cdot x_{r_3}(t^q)
  \cdot x_{r_4}(t^{q^2}),\\
x_3(t):&=x_{r_5^1}(t)=x_{r_6^1}(t^q)=x_{r_7^1}(t^{q^2})
=x_{r_5}(t)
  \cdot x_{r_6}(t^q)
  \cdot x_{r_7}(t^{q^2}),\\
x_4(t):&=x_{r_8^1}(t)=x_{r_{10}^1}(t^q)=x_{r_{9}^1}(t^{q^2})
=x_{r_8}(t)
 \cdot x_{r_{10}}(t^q)
 \cdot x_{r_{9}}(t^{q^2}).
\end{alignat*}
For $t\in \mathbb{F}_{q}$, set
$x_2(t):=x_{r_2^1}(t)
=x_{r_{2}}(t)$,
$x_5(t):=x_{r_{11}^1}(t)
=x_{r_{11}}(t)$,
$  x_6(t):=x_{r_{12}^1}(t)
=x_{r_{12}}(t)$.
Then the root subgroups of $^3D_4^{syl}(q^3)$ are
$X_i:=\{x_i(t) \mid t\in \mathbb{F}_{q^3}\}$ $(i=1,3,4)$
and $X_i:=\{x_i(t) \mid t^q=t,{\ }t\in \mathbb{F}_{q^3}\}$ $(i=2,5,6)$.
\end{Definition/Lemma}

\begin{Corollary}
$|D_4^{syl}(q^3)|=q^{36}$, $|A_8(q^3)|=q^{84}$ and
${{^3{D}}_4^{syl}}(q^3) \leqslant D_4^{syl}(q^3) \leqslant A_8(q^3)$.
\end{Corollary}
\begin{Definition}
A subgroup $P\leqslant {{^3{D}}_4^{syl}}(q^3)$ is a \textbf{pattern subgroup},
if it is generated by some root subgroups, i.e.
 $P=\left\langle {\,} X_i {\,}\middle|{\,} i\in I \subseteq \{1,2,\dots,6\}{\,}\right\rangle
 \leqslant {{^3{D}}_4^{syl}}(q^3)$.
\end{Definition}
We get the commutators of ${^3}D_4^{syl}(q^3)$ by calculation.
\begin{Lemma} \label{commutator-3D4}
Let $t_1,t_3,t_4 \in \mathbb{F}_{q^3}$, $t_2,t_5,t_6 \in \mathbb{F}_{q}$
and define the commutator
\begin{align*}
 [x_i(t_i), x_j(t_j)]:=x_i(t_i)^{-1}x_j(t_j)^{-1}x_i(t_i)x_j(t_j).
\end{align*}
Then the non-trivial commutators of ${^3}D_4^{syl}(q^3)$ are determined as follows:
\begin{align*}
[x_1(t_1), x_2(t_2)]=& x_3(-t_2t_1)x_4(t_2t_1^{q+1})x_5(-t_2t_1^{q^2+q+1})x_6(2t_2^2t_1^{q^2+q+1}),\\
[x_1(t_1), x_3(t_3)]=& x_4(t_1t_3^q+t_1^qt_3)
                       x_5(-t_1^{q+1}t_3^{q^2}-t_1^{q^2+q}t_3-t_1^{q^2+1}t_3^{q})
                      x_6(-t_1t_3^{q^2+q}-t_1^qt_3^{q^2+1}-t_1^{q^2}t_3^{q+1}),\\
[x_1(t_1), x_4(t_4)]=& x_5(t_1t_4^q+t_1^qt_4^{q^2}+t_1^{q^2}t_4),\\
[x_3(t_3), x_4(t_4)]=& x_6(t_3t_4^q+t_3^qt_4^{q^2}+t_3^{q^2}t_4),\\
[x_2(t_2), x_5(t_5)]=& x_6(t_2t_5).
\end{align*}
\end{Lemma}

\begin{Notation} \label{notation:x(t)-3D4}
 Set
$ x(t_1, t_2, t_3, t_4, t_5, t_6)
 := x_2(t_2)x_1(t_1)x_3(t_3)x_4(t_4)x_5(t_5)x_6(t_6)\in {^3{D}_4^{syl}(q^3)}$.
\end{Notation}

\begin{Proposition}[Sylow $p$-subgroup $^3{D}_4^{syl}(q^3)$]\label{sylow p-subg, 3D4}
A matrix Sylow $p$-subgroup
of the Steinberg triality group $^3{D}_4(q^3)$
is $^3{D}_4^{syl}(q^3)$
with
\begin{align*}
 {}^3{D}_4^{syl}(q^3)
 =&\left\{ x(t_1, t_2, t_3, t_4, t_5, t_6){\,}\middle|{\,}
 t_1, t_3, t_4 \in {\mathbb{F}_{q^3}},{\ } t_2, t_5, t_6 \in {\mathbb{F}_{q}}\right\},
\end{align*}
where
\begin{align*}
& x(t_1, t_2, t_3, t_4, t_5, t_6)\\
=&
 \left(
\newcommand{\mc}[3]{\multicolumn{#1}{#2}{#3}}
\begin{array}{cccccccc}\cline{2-7}
\mc{1}{c|}{1}
& \mc{1}{c|}{{t_1} }
& \mc{1}{c|}{ {-t_3} }
& \mc{1}{c|}{{
\begin{array}{l}
 t_1{t_3^q}\\
 +t_4
\end{array} } }
& \mc{1}{c|}{{
\begin{array}{l}
 t_1{t_3^{q^2}}\\
 +{t_4^{q^2}}
\end{array}}}
& \mc{1}{c|}{
\begin{array}{l}
 t_1{t_4^q}\\
+t_5
\end{array}}
& \mc{1}{c|}{\begin{array}{l}
 -t_1{t_3^{q^2+q}}
+t_3{t_4^q}\\
+t_6
 \end{array}}
& \rule{0pt}{22pt}
\begin{array}{l}
-t_1{t_3^q}{t_4^{q^2}}
-t_1{t_3^{q^2}}t_4\\
-t_1t_6
+t_3t_5
-{t_4^{q^2+1}}
\end{array}\\\cline{2-7}
×
& \mc{1}{c|}{1}
& \mc{1}{c|}{t_2}
& \mc{1}{c}{{
\begin{array}{l}
 {t_1^q}t_2\\
 +{t_3^q}
\end{array}
} }
& \mc{1}{c}{{
\begin{array}{l}
 {t_1^{q^2}}t_2\\
 +{t_3^{q^2}}
\end{array}
}}
& \mc{1}{c}{{
\begin{array}{l}
-{t_1^{q^2+q}}t_2\\
+{t_4^q}
\end{array}
}}
& \begin{array}{l}
-{t_1^q}t_2{t_3^{q^2}}
-{t_1^{q^2}}t_2{t_3^q}\\
-{t_2}{t_4^q}
-{t_3^{q^2+q}}
  \end{array}
& \rule{0pt}{35pt}
\begin{array}{l}
-{t_1^{q^2+q}}t_2t_3
-{t_1^q}{t_2}{t_4^{q^2}}\\
-{t_1^{q^2}}t_2t_4
-t_2t_5\\
-{t_3^q}{t_4^{q^2}}
-{t_3^{q^2}}t_4
-t_6
\end{array}\\\cline{3-3}
× & × & \mc{1}{c}{1} & \mc{1}{c}{{t_1^q} }
& \mc{1}{c}{{t_1^{q^2}}}
& {-{t_1^{q^2+q}}}
& \rule{0pt}{26pt}
\begin{array}{l}
   -{t_1^q}{t_3^{q^2}}
-{t_1^{q^2}}{t_3^q}\\
-{t_4^q}
  \end{array}
& \begin{array}{l}
  -{t_1^{q^2+q}}t_3
-{t_1^q}{t_4^{q^2}}\\
-{t_1^{q^2}}t_4
-t_5\\
  \end{array}\\
× & × & × & 1 & 0 & -{t_1^{q^2}} & -t_3^{q^2}
&  \rule{0pt}{17pt}
   -{t_1^{q^2}}t_3
   -{t_4^{q^2}} \\
× & × & × & × & 1 & -{t_1^q} &-t_3^q
& -{t_1^q}t_3
    -t_4\\
× & × & × & × & × & 1
& -t_2
& t_1t_2+t_3\\
× & × & × & × & × & × & 1
& -t_1\\
× & × & × & × & × & × & × & 1\\
\end{array}
\right)
\end{align*}
\end{Proposition}
\begin{proof}
We know that
${^3}D_4^{syl}(q^3)$ is a Sylow $p$-subgroup of ${^3}D_4(q^3)$.
By calculation,
we obtain the matrix form from \ref{root subgroups-3D4}.
\end{proof}
Let $t\in \mathbb{F}_{q}$,
$4 \leq  n\in \mathbb{N}^*$,
and
$ x_{i,j}(t): = I_{2n}+t e_{i,j}-t e_{2n-j,2n-i}
=\tilde{x}_{i,j}(t)\tilde{x}_{2n+1-j,2n+1-i}(-t) \in {D}_n^{syl}(q)$
for all $(i,j)\in \UP$.
Let
 $J:=\{(1,2), (1,3), (1,4), (1,5), (1,6), (1,7), (2,3)\} \subseteq \UP$.
We compare the Sylow $p$-subgroups $A_n(q)$,
$D_n^{syl}(q)$ and ${^3}D_4^{syl}(q^3)$.
\begin{Comparison}[Sylow $p$-subgroups]
\label{com:sylow-3D4}
\begin{itemize}
\setlength\itemsep{0em}
 \item [(1)] There exists an order of
$\prod_{(i,j)\in \UR}
            {\tilde{x}_{i,j}(t_{i,j})}$,
            such that for every matrix
            $u=(u_{i,j})\in A_n(q)$,
            we have $u_{k,l}=t_{k,l}$ for all $(k,l)\in \UR$.
 \item [(2)] There exists an order of
$\prod_{(i,j)\in \UP}
            {{x}_{i,j}(t_{i,j})}$,
            such that for every matrix
            $u=(u_{i,j})\in D_n^{syl}(q)$,
            we have $u_{k,l}=t_{k,l}$ for all $(k,l)\in \UP$
 (see \ref{sylow p-subg, D4}).
 \item [(3)] For ${^3}D_4^{syl}(q^3)$, the property does not hold.
 In \ref{sylow p-subg, 3D4}, for instance,
 we have matrix entries $t_1$, $t_2$ and
 up to sign also $t_3$ with positions in $J$,
 but $t_4$, $t_5$ and $t_6$
 appear in $J$ only in
 polynomials
 involving the other
 parameters.
\end{itemize}
\end{Comparison}

\section{The intermediate group $G_8(q^3)$}
\label{sec:bigger group}
In this section,
we construct a group $G_8(q^3)$
such that ${^3{D}}_4^{syl}(q^3)\leqslant G_8(q^3) \leqslant A_8(q^3)$.
Then we determine a monomial $G_8(q^3)$-module
to imitate the $D_n$ case in Section \ref{sec: monomial 3D4-module}.
In Section \ref{partition of U-3D4}, we use the group $G_8(q^3)$
to calculate the superclasses of ${^3}D_4^{syl}(q^3)$.

\begin{Definition/Lemma}[An intermediate group $G_8(q^3)$]
\label{larger group G8-3D4}
Let
\begin{align*}
G_8(q^3):=&
 \left\{
 u=(u_{i,j}) \in A_8(q^3)
{\,}\middle|{\,}
{\left\{\begin{array}{ll}
      u_{i,j}=0, &  (i,j)=(4,5)\\
      u_{i,j}\in \mathbb{F}_q, & (i,j)\in \{(2,3), (6,7)\}\\
      u_{i,j+1}=u_{i,j}^q \in \mathbb{F}_{q^3},    & (i,j)\in \{(2,4), (3,4)\}\\
      u_{i-1,j}=u_{i,j}^q \in \mathbb{F}_{q^3},    & (i,j)\in \{(5,6), (5,7)\}\\
       \end{array}
\right.}
\right\}.
\end{align*}
Then $G_8(q^3)$ is a subgroup of $A_8(q^3)$ of order $q^{65}$.
\end{Definition/Lemma}
\begin{proof}
By direct calculation.
\end{proof}

\begin{Notation}
Write
$\ddot{J}:=\UR\backslash\{(2,5),(3,5),(4,5), (4,6),(4,7)\}$.
For $(i,j)\in \ddot{J}$ and $t\in \mathbb{F}_{q^3}$, set
\begin{align*}
\dot{x}_{i,j}(t):=
{\left\{
\begin{array}{ll}
\tilde{x}_{i,j}(t) \tilde{x}_{i,j+1}(t^q) & \text{if }(i,j)\in \{(2,4),(3,4)\} \\
\tilde{x}_{i,j}(t) \tilde{x}_{i-1,j}(t^q) & \text{if }(i,j)\in \{(5,6),(5,7)\} \\
\tilde{x}_{i,j}(t) & \text{otherwise}
\end{array}
\right..}
\end{align*}
\end{Notation}

\begin{Proposition}\label{pairwise diff-G8q3}
$G_8(q^3)=\Bigg\{\prod_{(i,j)\in \ddot{J}}{\dot{x}_{i,j}(t_{i,j})}
{\,}\Bigg|{\,}
 t_{i,j}\in
{\left\{\begin{array}{ll}
       \mathbb{F}_q       & \text{if } (i,j)\in \{(2,3), (6,7)\}\\
       \mathbb{F}_{q^3}   & \text{otherwise}   \\
       \end{array}
\right.} \Bigg\}$,
where the product can be taken in an arbitrary, but fixed, order.
\end{Proposition}
\begin{proof}
Let $S$ denote the right side. Then $S\subseteq G_8(q^3)$ since $\dot{x}_{i,j}(t_{i,j})\in G_8(q^3)$.
Fix an order of the product,
and suppose that
$\prod_{(i,j)\in \ddot{J}}{\dot{x}_{i,j}(t_{i,j})}=\prod_{(i,j)\in \ddot{J}}{\dot{x}_{i,j}(s_{i,j})}$.
Then $t_{i,j}=s_{i,j}$ for all $(i,j)\in \ddot{J}$ by \ref{pairwise diff-3D4},
and $|S|=q^{65}$.
Hence $S=G_8(q^3)$ since $|G_8(q^3)|=q^{65}$.
\end{proof}

\begin{Corollary}
${{^3{D}}_4^{syl}}(q^3) \leqslant G_8(q^3) \leqslant A_8(q^3)$.
\end{Corollary}
\begin{Comparison}[Intermediate groups]
\label{com:intermediate group-3D4}
\begin{itemize}
 \setlength\itemsep{0em}
 \item [(1)] The intermediate group of $A_n(q)$ is $A_n(q)$ itself.
 \item [(2)] The intermediate group of $D_n^{syl}(q)$ is $A_{2n}(q)$
 (see \cite[3.1.2]{Markus1}).
 \item [(3)] The intermediate group of ${^3}D_4^{syl}(q^3)$ is $G_8(q^3)$
 (see \ref{larger group G8-3D4}).
\end{itemize}
\end{Comparison}


\section{Monomial ${{^3}D_4^{syl}(q^3)}$-module}
\label{sec: monomial 3D4-module}
Let $G:=G_8(q^3)$ and $U:={^3}D_4^{syl}(q^3)$.
In this section,
we construct an $\mathbb{F}_q$-subspace $V$ of $V_0$ (\ref{V, 3D4}),
a projection $\pi \colon V_0\to V$ (\ref{pi,3D4} and \ref{pi proj,3D4})
and a non-degenerate bilinear form ${\kappa_q}|_{V\times V}$ (\ref{kappa q, nond, 3D4}).
Then we determine an $\mathbb{F}_q$-linear group action $-\circ-$ (\ref{circ action-3D4})
and a surjective 1-cocycle $f$ of $G$ in $V$ (\ref{mono. lin.-3D4}).
Thus the monomial linearisation $(f,{\kappa_q}|_{V\times V})$ for $G_8(q^3)$
(\ref{mono. lin.-3D4}) is established.
Since $f|_U$ is a bijective 1-cocycle (\ref{f_U bij-3D4}),
$(f|_U,{\kappa_q}|_{V\times V})$ is a monomial linearisation for ${^3}D_4^{syl}(q^3)$ (\ref{mono. lin. for U-3D4}).
Then we make $\mathbb{C}U$ into a monomial $G_8(q^3)$-module (\ref{fund thm U-3D4}).

Let $V_0:=\mathrm{Mat}_{8\times 8}(q^3)$.
For any $I\subseteq \Squar$,
let $V_I:=\bigoplus_{(i,j)\in I}{\mathbb{F}_{q^3}}e_{ij} \subseteq V_0$.
In particular, $V_{\Squarb}=V_0$.
Then $V_I$ is an $\mathbb{F}_{q^3}$-vector space and also an $\mathbb{F}_{q}$-vector space.
We know
 $J=\{(1,2), (1,3), (1,4), (1,5), (1,6), (1,7), (2,3)\}$,
and $\dim_{\mathbb{F}_q}{V_J}=21$.
The \textbf{trace} of $A=(A_{i,j})\in V_0$ is denoted by
$\mathrm{tr}(A):=\sum_{i=1}^{8}{A_{i,i}}$.
The map
$
\kappa\colon V_0\times V_0  \to  {\mathbb{F}_{q^3}}: (A,B)\mapsto
 \mathrm{tr} (A^\top B)
$
is a symmetric $\mathbb{F}_{q^3}$-bilinear form on $V_0$
which is called the \textbf{trace form}.
In particular, $\kappa(A,B)=\sum_{(i,j)\in \Squarb}{A_{i,j}B_{i,j}}$
and $\kappa$ is non-degenerate.
 Let $V_J^{\bot}$ denote the orthogonal complement of $V_J$ in $V_0$
 with respect to the trace form $\kappa$, i.e.
 $V_J^{\bot}:=\{B\in V_0 \mid \kappa(A,B)=0,{\ }\forall {\ } A \in V_J\}$.
 Then
  $V_J^{\bot}=V_{\Squarb \backslash J}$
  {and}
  $V_0=V_J \oplus V_J^{\bot}$.
 $\kappa|_{V_J\times V_J}\colon V_J\times V_J \to  \mathbb{F}_{q^3}$ is a non-degenerate bilinear form.
The map
$\pi_J\colon V_0=V_J\oplus V_J^\bot  \to  V_J:
  A\mapsto
\sum_{(i,j)\in J}{A_{i,j}e_{i,j}}$
is a projection of $V_0$ to the first component $V_J$.
The \textbf{support}
of $A\in \mathrm{Mat}_{8\times 8}(K)$ is
$ \mathrm{supp}(A):= \{(i,j)\in \Squar \mid  A_{i,j}\neq 0\}$.
Let $V\subseteq V_0$ be a subspace of $V_0$, and set
$\mathrm{supp}(V):= \bigcup_{A\in V} \mathrm{supp}(A)$.
Suppose $A,B\in V_0$ such that $\mathrm{supp}(A)\cap\mathrm{supp}(B)\subseteq J$.
Then
$\kappa(A,B)=\kappa(\pi_J(A),B)=\kappa(A,\pi_J(B))
= \kappa(\pi_J(A),\pi_J(B))=\kappa|_{V_J\times V_J}(\pi_J(A),\pi_J(B))$.

The map $\phi_0 \colon  \mathbb{F}_{q^3} \to  \mathbb{F}_q: t \mapsto t+t^q+t^{q^2}$
is an $\mathbb{F}_q$-epimorphism
and $|{\ker}\phi_0|=q^2$ (see \cite[3.2]{sun2016arxiv}).
There exists an element $\eta\in {\mathbb{F}_{q^3}}\backslash {\mathbb{F}_{q}}$
(i.e. $\eta \in {\mathbb{F}_{q^3}}$ but $\eta \notin {\mathbb{F}_{q}}$)
such that
$\eta^{q^2}+\eta^{q}+\eta=1$
(see \cite[3.3]{sun2016arxiv}).
From now on, we fix an element $\eta\in {\mathbb{F}_{q^3}}\backslash {\mathbb{F}_{q}}$
such that $\eta^{q^2}+\eta^{q}+\eta=1$.
Then $1+\eta^{1-q^2}\neq 0$ (see \cite[3.4]{sun2016arxiv}).
\begin{Notation/Lemma}[3.5 of \cite{sun2016arxiv}] \label{pi_q, 3D4}
Let
$\pi_q \colon  {\mathbb{F}_{q^3}}  \to  {\mathbb{F}_{q}}:
x\mapsto \phi_0(\eta x)=(\eta x)^{q^2}+(\eta x)^q+{\eta x}$.
Then $\pi_q$ is an $\mathbb{F}_q$-epimorphism,
and $\mathbb{F}_{q^3}=\ker{\pi_q}\oplus \mathbb{F}_q$.
In particular, $\pi_q|_{\mathbb{F}_q}=\mathrm{id}_{\mathbb{F}_q}$
and $\pi_q^2=\pi_q$.
\end{Notation/Lemma}
\begin{Corollary}
If $x\in \mathbb{F}_{q^3}$.
Then
$\pi_q(xy)=0$ for all $y\in \mathbb{F}_{q^3}$
if and only if
$x=0$.
\end{Corollary}

\begin{Proposition} \label{kappa_q,3D4}
The map
$\kappa_q:=\pi_q\circ\kappa \colon  V_0\times V_0  \to  \mathbb{F}_q
 :(A,B)\mapsto \pi_q\circ\kappa(A,B)=\pi_q\big(\mathrm{tr} (A^\top B) \big)$
 is a symmetric ${\mathbb{F}_q}$-bilinear form on $V_0$.
\end{Proposition}
\begin{proof}
We know that $\kappa$ is a symmetric bilinear map,
and that $\pi_q$ is an $\mathbb{F}_q$-epimorphism (\ref{pi_q, 3D4}),
then the claim is proved.
\end{proof}

\begin{Notation/Lemma}\label{V, 3D4}
Let
\begin{align*}
V:&=\left\{  A=(A_{ij})\in V_0
{\,}\middle|{\,}
\mathrm{supp}(A)\in J,
{\ }A_{16}, {\ }A_{17}, {\ }A_{23}\in {\mathbb{F}_q}, {\ }A_{14}=A_{15}^q\in \mathbb{F}_{q^3}  \right\}\\
&=\left\{
\left(
{%
\newcommand{\mc}[3]{\multicolumn{#1}{#2}{#3}}
\begin{array}{cccccccc}\cline{2-7}
\mc{1}{c|}{} & \mc{1}{c|}{A_{12}} & \mc{1}{c|}{A_{13}}
& \mc{1}{c|}{\rule{0pt}{12pt} A_{15}^q}
& \mc{1}{c|}{A_{15}} & \mc{1}{c|}{A_{16}} & \mc{1}{c|}{A_{17}} & \\\cline{2-7}
× & \mc{1}{c|}{} & \mc{1}{c|}{A_{23}} & × & × & ×&  & ×\\\cline{3-3}
\end{array}
}%
\right)_{8\times 8}
\in V_0
{\,}\middle|{\,}
{\left\{
\begin{array}{l}
 A_{12}, A_{13}, A_{15} \in {\mathbb{F}_{q^3}}\\
 A_{16}, A_{17}, A_{23}\in {\mathbb{F}_q}
\end{array}
\right.}
\right\}
\end{align*}
omitting all zero entries in the matrices,
in particular at positions $(1,1)$ and $(1,8)$.
Then $V$ is a 12-dimensional subspace of $V_J$
over $\mathbb{F}_{q}$
and $\mathrm{supp}(V)=J$.
\end{Notation/Lemma}

We define the following map $\pi$,
which plays a crucial role in our later statement.
\begin{Notation/Lemma}\label{pi,3D4}
Let
\begin{align*}
& \pi \colon V_0 \to  V:
  A\mapsto
{
\begin{array}{l}
A_{12}e_{12}+A_{13}e_{13}
+(\frac{A_{14}^{q^2}+A_{15}\eta^{1-q^2}}{1+\eta^{1-q^2}})^qe_{14}+{\frac{A_{14}^{q^2}+A_{15}\eta^{1-q^2}}{1+\eta^{1-q^2}}}e_{15}\\
+\pi_q(A_{16})e_{16}+\pi_q(A_{17})e_{17}+\pi_q(A_{23})e_{23}
\end{array}
},
\end{align*}
i.e.
\begin{align*}
 \pi(A)=&
\left(
{%
\newcommand{\mc}[3]{\multicolumn{#1}{#2}{#3}}
\begin{array}{cccccccc}\cline{2-7}
\mc{1}{c|}{}
& \mc{1}{c|}{A_{12}}
& \mc{1}{c|}{A_{13}}
& \mc{1}{c|}{
\rule{0pt}{19pt}
\begin{array}{l}
(\frac{A_{14}^{q^2}+A_{15}\eta^{1-q^2}}{1+\eta^{1-q^2}})^q\\
 \end{array}}
& \mc{1}{c|}{
\rule{0pt}{19pt}
{\frac{A_{14}^{q^2}+A_{15}\eta^{1-q^2}}{1+\eta^{1-q^2}}}}
& \mc{1}{c|}{\pi_q(A_{16})}
& \mc{1}{c|}{\pi_q(A_{17})}
& \\\cline{2-7}
× & \mc{1}{c|}{}
& \mc{1}{c|}{
\begin{array}{l}
\pi_q(A_{23})\\
\end{array}}
& × & × & ×&  & ×\\\cline{3-3}
\end{array}
}%
\right)
\end{align*}
Then $\pi$ is an ${\mathbb{F}_q}$-epimorphism.
In particular,
$\pi|_V=\mathrm{id}_{V}$, $\pi^2=\pi$ and $\pi(I_8)=O_8$.
\end{Notation/Lemma}
\begin{proof}
Let $A=(A_{ij}), B=(B_{ij})\in V_0$ and $k\in \mathbb{F}_q$. Then
we obtain
$\pi(A+B)=\pi(A)+\pi(B)$ and  $\pi(k A)=k\pi(A)$
by straightforward calculation.
By \ref{pi_q, 3D4} and \ref{V, 3D4}, $\pi|_V=\mathrm{id}_V$.
Thus $\pi$ is an ${\mathbb{F}_q}$-epimorphism.
\end{proof}

\begin{Lemma}\label{VT,3D4}
Let $V^{\bot}$ denote the orthogonal complement of the subspace $V$ of $V_0$
 with respect to $\kappa_q$, i.e.
$V^\bot:= \{B\in V_0 \mid \kappa_q(A,B)=0 {\ } \text{ for all }  A \in V\}$,
 and
 \begin{align*}
W:=& \bigoplus_{(i,j)\notin J}{\mathbb{F}_{q^3}e_{ij}}
+{\ker{\pi_q}}e_{16}+{\ker{\pi_q}} e_{17}+{\ker{\pi_q}} e_{23}
+\{xe_{15}-x^q\eta^{q-1}e_{14} \mid x\in \mathbb{F}_{q^3}\}\\
=& \{ A=(A_{ij})\in V_0
   \mid A_{12}=A_{13}=0,A_{14}=-A_{15}^q\eta^{q-1}\in \mathbb{F}_{q^3}, A_{16}, A_{17}, A_{23} \in \ker\pi_q \}.
\end{align*}
Then
$W =V^\bot$.
\end{Lemma}

\begin{proof}
For all $B\in W$ and $A\in V$,
$\kappa_q(A,B)
 =\pi_q\kappa(A,B)
{=}\pi_q(A_{15}B_{15}-A_{15}^qB_{15}^q\eta^{q-1})
 =0$,
so $ B \in V^\bot$ for all $B\in W$
and $W\subseteq V^\bot$.

For all $B=(B_{ij})\in V^\bot \subseteq V_0$ and $ A=(A_{ij})\in V$,
we have
$0=\kappa_q(A,B)
 =\pi_q\kappa(A,B)
 =\pi_q(A_{12}B_{12}+A_{13}B_{13}+A_{15}^qB_{14}+A_{15}B_{15}+A_{16}B_{16}+A_{17}B_{17}+A_{23}B_{23})$.
Assume that $B_{12}\neq 0$.
If $A=A_{12}e_{12}=B_{12}^{-1}e_{12}$,
then $0=\kappa_q(A,B)=\pi_q(A_{12}B_{12})=\pi_q(1)=1$,
this is a contradiction, so $B_{12}=0$.
Similarly, $B_{13}=0$.
If $A=A_{16}e_{16}=e_{16}$, then
$ 0=\kappa_q(A,B)=\pi_q(A_{16}B_{16})=\pi_q(B_{16})$,
so $B_{16}\in \ker\pi_q$.
Similarly, $B_{17}, B_{23}\in \ker\pi_q$.
Assume $B_{14}\neq-B_{15}^q\eta^{q-1}$, $\mathrm{i.e.}$ $B_{14}^{q^2}\eta^{q^2}+B_{15}\eta \neq0$.
If $A=A_{15}^qe_{14}+A_{15}e_{15}$ $(A_{15}\in \mathbb{F}_{q^3})$, then
\begin{align*}
  0&=\kappa_q(A,B)
  =\pi_q(A_{15}^qB_{14}+A_{15}B_{15})
  = \phi_0\big(A_{15}(B_{14}^{q^2}\eta^{q^2}+B_{15}\eta)\big).
\end{align*}
Since $B_{14}^{q^2}\eta^{q^2}+B_{15}\eta \neq 0$,
we have $\ker\phi_0 =\mathbb{F}_{q^3}$
and
$|\ker\phi_0| = {q^3}>q^2$.
This is a contradiction since $ |\ker \phi_0| = {q^2}$. Thus $B_{14}=-B_{15}^q\eta^{q-1}$.
Hence $B \in W$ and $ V^\bot \subseteq W$.
Therefore, $ V^\bot=W$.
\end{proof}

\begin{Corollary}\label{pi proj,3D4}
 $V_0=V\oplus V^\bot$,
 and
 $\pi \colon V_0 \to  V$ is the projection to the first component $V$.
\end{Corollary}
\begin{proof}
We know that $V+ V^\bot \subseteq V_0$
since $V$ and $V^\bot$ are $\mathbb{F}_q$-subspaces of $V_0$.
If $A=(A_{ij})\in V\cap V^\bot$,
then $A_{15}^q
      \stackrel{A\in V}{=} A_{14}
      \stackrel{A\in V^\bot}{=} -A_{15}^q\eta^{q-1}$,
    so $A_{15}^q(1+\eta^{q-1})=0$ and $A_{15}=0$.
Thus $A=0$ by \ref{pi_q, 3D4} and \ref{VT,3D4}.
Hence $V\cap  V^\bot=\{0\}$.
If $A\in V_0$, then $A=\pi(A)+(A-\pi(A))$ and $\pi(A)\in V$.
It is enough to show that $A-\pi(A)\in V^\bot$.
Let $B=(B_{i,j})=\pi(A)\in V$ and $C=(C_{i,j})=A-\pi(A)$.
By \ref{pi_q, 3D4} we have $\mathbb{F}_{q^3}=\ker \pi_q \oplus \mathbb{F}_q$.
Thus it is sufficient to prove that $C_{14}=-C_{15}^q\eta^{q-1}$.
We have
$C_{15}=\frac{A_{15}-A_{14}^{q^2}}{1+\eta^{1-q^2}}$
and
$C_{14}={\frac{A_{14}-A_{15}^q}{1+\eta^{q-1}}}\eta^{q-1}
       =-C_{15}^q\eta^{q-1}$.
Hence $V_0=V\oplus V^\bot$.
\end{proof}

\begin{Corollary}\label{kappa q, nond, 3D4}
$\kappa_q|_{V\times V}$ is a non-degenerate $\mathbb{F}_q$-bilinear form.
\end{Corollary}

\begin{Lemma}\label{pi=piJ, 3D4}
$\pi=\pi\circ \pi_J$.
In particular, if $A\in V_0$ and $\pi_J(A)\in V$, then $\pi(A)=\pi_J(A)$.
\end{Lemma}
\begin{proof}
Let $A\in V_0$,
the $\pi(A)$ depends on the entries of $\{A_{i,j}{\,|\,} (i,j)\in J\}$ by \ref{pi_q, 3D4},
thus $\pi(A)=\pi\circ \pi_J(A)$.
Let $\pi_J(A)\in V$.
Then $\pi(A)=\pi\circ \pi_J(A) = \pi(\pi_J(A)) = \pi_J(A)$
since $\pi|_V=\mathrm{id}_V$.
\end{proof}

\begin{Corollary}\label{kappa q, pi-3D4}
Suppose $A,B\in V_0$,
such that $\mathrm{supp}(A)\cap\mathrm{supp}(B)\subseteq J$ and $\pi_J(A)\in V$.
Then
\begin{align*}
 \kappa_q(A,B)=\kappa_q(\pi(A),B)=\kappa_q(A,\pi(B))
= \kappa_q(\pi(A),\pi(B))=\kappa_q|_{V\times V}(\pi(A),\pi(B)).
\end{align*}
\end{Corollary}
\begin{proof}
We have
$
\kappa_q(A,B)
{=} \pi_q\circ \kappa (A,B)
{=}\pi_q\circ \kappa (\pi_J(A),B)=\kappa_q(\pi_J(A),B)
 \stackrel{\ref{pi=piJ, 3D4}}{=} \kappa_q(\pi(A),B)\\
 \stackrel{\ref{pi proj,3D4}}{=}\kappa_q(\pi(A),\pi(B)){=}\kappa_q|_{V\times V}(\pi(A),\pi(B))
 \stackrel{\pi(B)\in V}{=} \kappa_q(A,\pi(B))$.
\end{proof}

\begin{Lemma}\label{pi_J,AgT,3D4}
Let  $A\in V$ and $g\in G$.
Then
$\pi_J(Ag^\top)\in V$.
In particular, $\pi_J(Ag^\top)=\pi(Ag^\top)$.
\end{Lemma}
\begin{proof}
Let $A\in V$ and $g\in G$. It is sufficient to prove that
 ${(Ag^\top)_{15}^q=(Ag^\top)_{14}}$
 and
 ${(Ag^\top)_{ij}\in \mathbb{F}_q}$
 for all $(i,j)\in \{(2,3), (1,6), (1,7)\}$.
We have
$(Ag^\top)_{15}^q
                       =(\sum_{j=1}^{8} A_{1j}g_{5j})^q
                =A_{15}^q+A_{16}g_{56}^q+A_{17}g_{57}^q
\stackrel{\ref{larger group G8-3D4} \& \ref{V, 3D4}}
{=} A_{14}+A_{15}g_{45}+A_{16}g_{46}+A_{17}g_{47}
= (Ag^\top)_{14}$,
and
 ${(Ag^\top)_{16}}
 =A_{16}+A_{17}g_{67}=A_{16}-A_{17}g_{23}\in\mathbb{F}_{q}$,
 ${(Ag^\top)_{17}}=A_{17}\in\mathbb{F}_{q}$,
 ${(Ag^\top)_{23}}=A_{23}\in\mathbb{F}_{q}$.
Thus $\pi_J(Ag^\top)\in V$.
\end{proof}

\begin{Lemma}\label{intersection in J, 3D4}
Let $A,B\in V$ and $g,h\in A_8(q^3)$.
Then we have that
\begin{align*}
\mathrm{supp}(Bh^\top)\cap \UR \subseteq  J
\quad \text{ and } \quad
\mathrm{supp}(Bh^\top)\cap \mathrm{supp}(Ag)\subseteq  J.
\end{align*}
\end{Lemma}
\begin{proof}
Every $h\in A_8(q^3)$ acts on $V$ from the right by a sequence of
elementary column operations
from left to right,
and $h^\top$ acts on $V$ from the right by a series of
elementary column operations from right to left.
Then the statements are obtained.
\end{proof}

\begin{Corollary}
Let $A,B\in V$ and $g\in G$.
Then
\begin{align*}
 \kappa_q(A, Bg)=\kappa_q(A, \pi(Bg))=\kappa_q(Ag^\top, B)=\kappa_q(\pi(Ag^\top), B).
\end{align*}
\end{Corollary}

\begin{proof}
If $A,B\in V$ and $g\in G$, then
$\kappa_q(A, Bg)
 \stackrel{\substack{ \ref{kappa q, pi-3D4}
                    \,\&\, \ref{intersection in J, 3D4}}}
                    {=} \kappa_q(A, \pi(Bg))$,
and
$ \kappa_q(A, Bg)
 =\kappa_q(Ag^\top, B)
 \stackrel{\substack{ \ref{kappa q, pi-3D4}
                    \,\&\, \ref{pi_J,AgT,3D4}\\
                    \&\, \ref{intersection in J, 3D4}}}
                    {=}\kappa_q(\pi(Ag^\top), B)$.
\end{proof}

\begin{Proposition}[Group action of $G$ on $V$]\label{circ action-3D4}
The map
\begin{align*}
-\circ- \colon V\times G \to  V: (A,g)\mapsto A\circ g:=\pi(Ag)
\end{align*}
is a group action,
and the elements of the group $G$ act as
$\mathbb{F}_q$-automorphisms.
\end{Proposition}

\begin{proof}
Let $A, B\in V$, $g, h\in G$ and $c\in \mathbb{F}_q$.
Since $\pi$ is $\mathbb{F}_q$-linear,
it is enough to prove
$A\circ (gh) {=}(A\circ g) \circ h$.
We have
\begin{align*}
 &  \kappa_q(B, A\circ(gh))
\stackrel{\ref{kappa q, pi-3D4}}{=}\kappa_q(B, A(gh))
= \kappa_q(Bh^\top, Ag)
\stackrel{\substack{ \ref{kappa q, pi-3D4}
                    \,\&\, \ref{pi_J,AgT,3D4}\\
                    \&\, \ref{intersection in J, 3D4}}}
                    {=} \kappa_q(\pi(Bh^\top), Ag)\\
\stackrel{\ref{kappa q, pi-3D4}}{=}&
                 \kappa_q(\pi(Bh^\top), A\circ g)
\stackrel{\ref{kappa q, pi-3D4}}{=} \kappa_q(Bh^\top, A\circ g)
                 =\kappa_q(B, (A\circ g)h)
\stackrel{\ref{kappa q, pi-3D4}}{=}\kappa_q(B, (A\circ g)\circ h).
\end{align*}
By \ref{kappa q, nond, 3D4}, $A\circ (gh) {=}(A\circ g) \circ h$.
Thus the proof is completed.
\end{proof}

\begin{Corollary}\label{circ g-circ gT-3D4}
If $A,B\in V$ and $g\in G$, then
$\kappa_q(A, B\circ g)=\kappa_q(A, Bg)=\kappa_q(Ag^\top, B)=\kappa_q(\pi(A g^\top), B)$.
\end{Corollary}

Henceforth denote $\pi(A g^{-\top})$  ($A\in V$, $g\in G$) by $A.g$.
Then this is a group action of $G$ by \ref{circ action-3D4}.
By \cite[\S 2.1]{Markus1}, we get the following result:
\begin{Corollary}\label{action A dot g-3D4}
There exists an unique linear action $-.-$ of $G$ on $V$:
\begin{align*}
-.- \colon V\times G  \to  V: (A,g)\mapsto A.g=\pi(A g^{-\top})
\end{align*}
such that
$\kappa_q|_{V\times V}(A.g,B)=\kappa_q|_{V\times V}(A,B\circ g^{-1})$
 for all  $B\in V$.
\end{Corollary}
\begin{Notation}  \label{f-3D4}
 Set $f:=\pi|_G \colon G \to  V$.
\end{Notation}

\begin{Lemma}\label{f(x)g,3D4}
Let $x, g \in G$, $1:=1_G=I_8$ and $0:=O_8=O_{8\times 8}$. Then
$f(x)g\equiv(x-1)g \mod V^\bot$.
In particular, $f(x)\equiv x-1 \mod V^\bot$.
\end{Lemma}

\begin{proof}
Let $x,g\in G$.
We have that $f(x)=\pi(x)
 \stackrel{\pi(1)=0}{=}\pi(x)-\pi(1)
\stackrel{\pi\text{ linear}}{=}\pi(x-1)$,
so $f(x)\equiv x-1 \mod V^\bot$.
For all $ A \in V$,
$\kappa_q(A, f(x)g)=\kappa_q(Ag^\top, f(x))
\stackrel{\ref{kappa q, pi-3D4}}{=}\kappa_q(\pi(Ag^\top), f(x))
{=}\kappa_q(\pi(Ag^\top), x-1)
\stackrel{\ref{intersection in J, 3D4}}{=}\kappa_q(Ag^\top, x-1)
    {=}\kappa_q(A, (x-1)g)$.
By \ref{kappa q, nond, 3D4},
$f(x)g\equiv(x-1)g \mod V^\bot$.
\end{proof}

\begin{Proposition}\label{f(xg)-3D4}
If $x, g \in G$, then
$ f(xg)=f(x)\circ g+f(g)$.
\end{Proposition}

\begin{proof}
Let $x, g \in G$.
For all $A\in V$,
$\kappa_q(A,f(xg))
    \stackrel{\ref{f(x)g,3D4}}{=}  \kappa_q(A,xg-1)
    = \kappa_q(A,(x-1)g+(g-1))
  \stackrel{\ref{f(x)g,3D4}}{=}
   \kappa_q(A,f(x)g+f(g))
    \stackrel{\ref{kappa q, pi-3D4}}{=}\kappa_q(A,\pi(f(x)g)+f(g))
    \stackrel{\ref{kappa q, pi-3D4}}{=}\kappa_q(A,f(x)\circ g+f(g))$.
Thus  $f(xg)=f(x)\circ g+f(g)$ by \ref{kappa q, nond, 3D4}.
\end{proof}


\begin{Proposition}[Bijective 1-cocycle of ${^3{D}_4^{syl}(q^3)}$]\label{f_U bij-3D4}
Let $U={^3{D}_4^{syl}(q^3)}$.
Then $f|_U:=\pi|_U \colon  U \to  V $ is a bijection.
In particular, $f|_U$ is a bijective 1-cocycle of $U$ in $V$.
\end{Proposition}
\begin{proof}
 Let $x:=x(t_1, t_2, t_3, t_4, t_5, t_6)\in U$.
 Then we have that
\begin{align*}
 & f|_U(x)
=\pi(x(t_1, t_2, t_3, t_4, t_5, t_6))\\
=&\left(
\newcommand{\mc}[3]{\multicolumn{#1}{#2}{#3}}
\begin{array}{cccccccc}\cline{2-7}
\mc{1}{c|}{}
& \mc{1}{c|}{t_1}
& \mc{1}{c|}{-t_3}
& \mc{1}{c|}{\begin{array}{l}
(\frac{t_1^{q^2}t_3+t_1t_3^{q^2}\eta^{1-q^2}}{1+\eta^{1-q^2}})^q\\
+t_4\\
 \end{array}}
& \mc{1}{c|}{\begin{array}{l}
\frac{t_1^{q^2}t_3+t_1t_3^{q^2}\eta^{1-q^2}}{1+\eta^{1-q^2}}\\
+t_4^{q^2}
 \end{array}}
& \mc{1}{c|}{\begin{array}{l}
\pi_q(t_1{t_4^q})\\
+t_5
 \end{array}}
& \mc{1}{c|}{
\rule{0pt}{28pt}
\begin{array}{l}
\pi_q( -t_1{t_3^{q^2+q}}\\
      {\quad}+t_3{t_4^q})\\
+t_6
 \end{array}}
& \\\cline{2-7}
× & \mc{1}{c|}{}
& \mc{1}{c|}{
\begin{array}{l}
t_2\\
\end{array}}
& × & × & ×& & ×\\\cline{3-3}
\end{array}
\right)
\end{align*}
Since $f\colon G\to V$ is well defined, $f|_U$ is well defined.
For all $A=(A_{ij})\in V$,
there exists an element $x:=x(t_1, t_2, t_3, t_4, t_5, t_6)\in U$ such that
$f|_U(x)=A$,
where
$t_1=A_{12}$,
$t_2=A_{23}\in {\mathbb{F}_{q}}$,
$t_3=-A_{13}$,
$t_4=A_{15}^q+\frac{A_{12}A_{13}^q+A_{12}^qA_{13}\eta^{q-1}}{1+\eta^{q-1}}$,
$t_5=A_{16}-\pi_q(A_{12}(A_{15}^{q^2}+\frac{A_{12}^qA_{13}^{q^2}+A_{12}^{q^2}A_{13}^q\eta^{q^2-q}}{1+\eta^{q^2-q}}))\in {\mathbb{F}_{q}}$
and
   $t_6=A_{17}+
      \pi_q(A_{12}A_{13}^{q^2+q}+A_{13}(A_{15}^{q^2}+\frac{A_{12}^qA_{13}^{q^2}+A_{12}^{q^2}A_{13}^q\eta^{q^2-q}}{1+\eta^{q^2-q}}))
      \in {\mathbb{F}_{q}}$.
Thus $f|_U$ is surjective.
Since $|U|=q^{12}=|V|$,
 $f|_U$ is bijective.
Hence
$f|_U$ is a bijective 1-cocycle of $U$ in $V$
by \ref{f(xg)-3D4}.
\end{proof}

\begin{Corollary}
\label{mono. lin.-3D4}
 $f=\pi|_G  \colon G \to  V$ is a surjective 1-cocycle of $G$ in $V$,
 and $(f,{\kappa_q}|_{V\times V})$ is a monomial linearisation for $G=G_8(q^3)$.
\end{Corollary}

\begin{Corollary}
\label{mono. lin. for U-3D4}
 $(f|_{^3{D}_4^{syl}(q^3)},{\kappa_q}|_{V\times V})$
is a monomial linearisation
for ${^3{D}_4^{syl}(q^3)}$.
\end{Corollary}

Now we  establish the monomial $G$-module $\mathbb{C}{\left(^3D_4^{syl}(q^3)\right)}$,
which is essential for the construction of the supercharacter theory for ${^3D_4^{syl}(q^3)}$.
\begin{Theorem}[Fundamental theorem for ${^3D_4^{syl}(q^3)}$]\label{fund thm U-3D4}
Let $U={^3D_4^{syl}(q^3)}$, $G=G_8(q^3)$ and
 \begin{align*}
  [A]=\frac{1}{|U|}\sum_{u\in U}{\overline{\chi_A(u)}u}
  \qquad \text{for all } A\in V,
 \end{align*}
 where $\chi_{A}(u)=\vartheta\kappa_q(A,f(u))$.
Then the set $\{[A] \mid A\in V\}$
forms a $\mathbb{C}$-basis for the complex group algebra $\mathbb{C}U$.
For all $g\in G,{\ }A\in V$,
let $[A]*g:=\chi_{A.g}(g)[A.g]=\vartheta\kappa_q(A.g,f(g))[A.g]$.
Then $\mathbb{C}U$ is a monomial $\mathbb{C}G$-module.
The restriction of the $*$-operation to $U$ is given by the usual right multiplication
  of $U$ on $\mathbb{C}U$, i.e.
\begin{align*}
 [A]*u=[A]u=\frac{1}{|U|}\sum_{y\in U}{\overline{\chi_A(y)}yu}
 \qquad \text{for all $u\in U,{\ } A\in V$}.
\end{align*}
\end{Theorem}
\begin{proof}
By \ref{mono. lin.-3D4}, $(f,{\kappa_q}|_{V\times V})$ is a monomial linearisation for $G$,
 satisfying that $f|_U$ is a bijective map.
By \ref{action A dot g-3D4}, $A.u:=\pi(Au^{-\top})$.
Thus the whole theorem is proved
in view of
\cite[2.1.35]{Markus1}.
\end{proof}

\begin{Comparison}[Monomial linearisations]
\label{com:monomial modules-3D4}
Let $U$ be $A_n(q)$, $D_n^{syl}(q)$ or ${^3}D_4^{syl}(q^3)$,
$G$ an intermediate group of $U$,
$V_0:=V_{\Squarb}$,
$V$ a subspace of $V_0$,
$J:=\mathrm{supp}(V)$,
$f\colon G\to V$ a surjective 1-cocycle of $G$
such that $f|_U$ is injective,
$\kappa\colon V \times V\to \mathbb{F}_q \text{ (or $\mathbb{F}_{q^3}$)}$
a trace form
such that
$(f,\kappa|_{V\times V})$ is a monomial linearisation for $G$
(i.e. $(f|_U,\kappa|_{V\times V})$ is a monomial linearisation for $U$).
Then the
corresponding notations for
$A_n(q)$ (see \cite[2.2]{Markus1}),
$D_n^{syl}(q)$ (see \cite[3.1]{Markus1})
and ${^3}D_4^{syl}(q^3)$ (see \S \ref{sec: monomial 3D4-module})
are listed as follows:
\begin{align*}
\begin{array}{|l|l|l|l|l|l|l|}
\hline
\multicolumn{1}{|c|}{U}
& \multicolumn{1}{c|}{G}
& \multicolumn{1}{c|}{V_0}
& \multicolumn{1}{c|}{J}
& \multicolumn{1}{c|}{V}
& \multicolumn{1}{c|}{f\colon G\to V}
& \multicolumn{1}{c|}{\kappa|_{V\times V}} \\\hline
A_n(q)
& A_{n}(q)
& \mathrm{Mat}_{n\times n}(q)
& \UR
& V=V_{\URb}
& f(g)=\pi_{\URb}(g)=g-I_n
& \kappa|_{V\times V}
\\\hline
D_n^{syl}(q)
& A_{2n}(q)
& \mathrm{Mat}_{2n\times 2n}(q)
& {\UP}
& V=V_{\UPb}
& f(g)=\pi_{\UPb}(g)
& \kappa|_{V\times V}
\\\hline
{^3}D_4^{syl}(q^3)
& G_8(q^3)
& \mathrm{Mat}_{8\times 8}(q^3)
& J
& V\neq V_J
& f(g)=\pi(g)\neq \pi_J(g)
& \kappa_q|_{V\times V}
\\\hline
\end{array}
\end{align*}
\end{Comparison}
From now on, we mainly consider the regular right module $(\mathbb{C}U,*)_{\mathbb{C}U}=\mathbb{C}U_{\mathbb{C}U}$.


\section{${{^3}D}^{syl}_4{(q^3)}$-orbit modules}
\label{sec:U-orbit modules-3D4}

Let $U:={{^3}D}^{syl}_4{(q^3)}$, $A\in V$ and $x_i(t_i)\in U$
($i=1,2,\dots,6$, $t_1, t_3, t_4 \in \mathbb{F}_{q^3}$, $t_2, t_5, t_6 \in \mathbb{F}_{q}$).
In this section,
we classify the $U$-orbit modules (\ref{prop:class orbit-3D4}),
and determine the stabilizers $\mathrm{Stab}_U(A)$ for all $A\in V$ (\ref{prop: 3D4-stab}).

If $A\in V$, then the $U$\textbf{-orbit module}
 associated to $A$ is
$\mathbb{C}\mathcal{O}_U([A])
:=\mathbb{C}\{[A]u \mid u\in U\}
=\mathbb{C}\{[A.u] \mid u\in U\}
=\mathbb{C}\text{-span}\left\{[C]\mid C\in \mathcal{O}_U(A)\right\}$,
where
$\mathcal{O}_U(A):=\left\{A.g \mid  g\in U\right\}$
is the \textbf{orbit} of $A$ under the operation $-.-$ defined in \ref{action A dot g-3D4}.
Similarly, we define the $G$\textbf{-orbit module}
$\mathbb{C}\mathcal{O}_G([A])$ for all $A\in V$.
The \textbf{stabilizer}
 $\mathrm{Stab}_U(A)$ of $A$ in $U$ is
$
 \mathrm{Stab}_U(A)=\{u\in U  \mid  A.u=A\}$,
and
 $\mathrm{dim}_{\mathbb{C}}\mathbb{C}\mathcal{O}_U([A])
 =|\mathcal{O}_U(A)|
 =\frac{|U|}{|\mathrm{Stab}_U(A)|}$.
Let $A,B\in V$.
Then
$\mathbb{C}\mathcal{O}_U([A])$ and $\mathbb{C}\mathcal{O}_U([B])$
are identical (if $A.u=B$ for some $u\in U$)
or their intersection is $\{0\}$.
 Two $\mathbb{C}U$-modules having no nontrivial
 $\mathbb{C}U$-homomorphism between them are called
 \textbf{orthogonal}.
\begin{Lemma}\label{3D4-A.xi, figures}
Let $A\in V$ and $x_i(t_i)\in U$ with $i\in\{1,2,\dots,6\}$,
$t_1, t_3, t_4 \in \mathbb{F}_{q^3}$ and $t_2, t_5, t_6 \in \mathbb{F}_{q}$.
Then $A.x_i(t_i)$  and
the corresponding figures of moves
are obtained as follows:
\begin{alignat*}{2}
A.x_1(t_1)=& A.\left(x_{34}(t_1^q)x_{35}(t_1^{q^2})\right),
& \qquad
A.x_2(t_2)=& A.x_{23}(t_2),\\
A.x_3(t_3)=& A.\left(x_{24}(t_3^q)x_{25}(t_3^{q^2})\right),
& \qquad
A.x_4(t_4)=& A.x_{26}(t_4^q)\\
A.x_5(t_5)=& A,
& \qquad
A.x_6(t_6)=& A.
\end{alignat*}

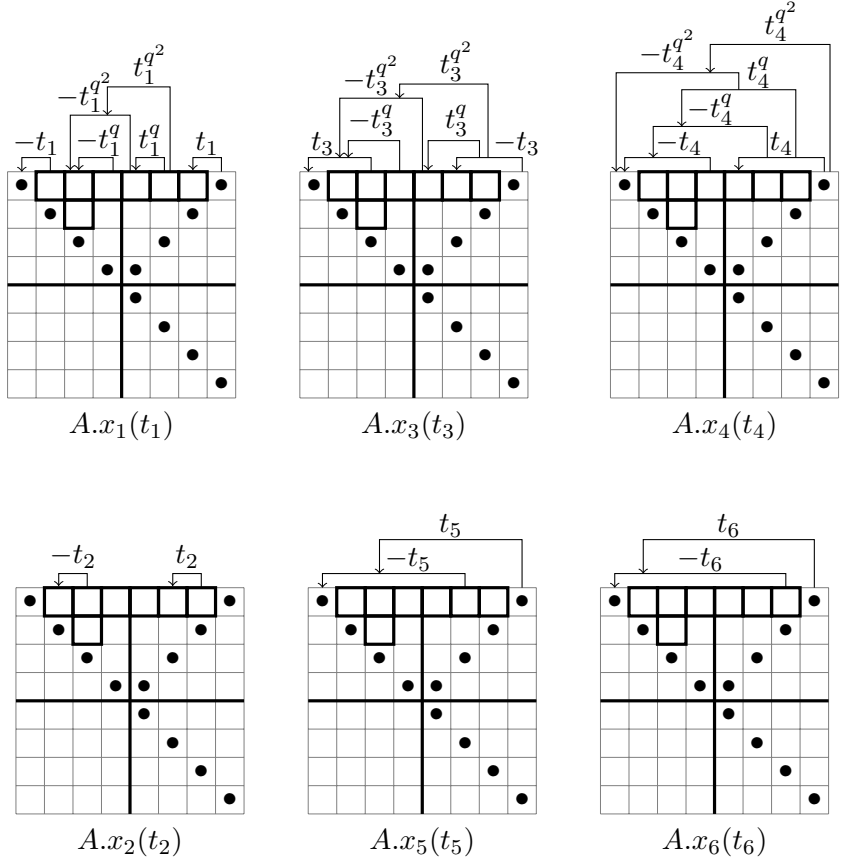
\begin{align*}
\begin{tikzpicture}[scale=0.75]
\draw[step=0.5cm, gray, very thin](-2, -2)grid(2, 2);
\draw[very thick] (-2,0)--(2,0);
\draw[very thick] (0,-2)--(0,2);
\draw (-1.75,1.75) node{$\bullet$};
\draw (-1.25,1.25) node{$\bullet$};
\draw (-0.75,0.75) node{$\bullet$};
\draw (-0.25,0.25) node{$\bullet$};
\draw (0.25,-0.25) node{$\bullet$};
\draw (0.75,-0.75) node{$\bullet$};
\draw (1.25,-1.25) node{$\bullet$};
\draw (1.75,-1.75) node{$\bullet$};
\draw (1.75,1.75) node{$\bullet$};
\draw (1.25,1.25) node{$\bullet$};
\draw (0.75,0.75) node{$\bullet$};
\draw (0.25,0.25) node{$\bullet$};
\foreach \x in {-1.5, -1,...,1.0}
{
\draw [very thick] (\x,1.5)   rectangle +(0.5,0.5);
}
\draw [very thick] (-1.0,1.0) rectangle +(0.5,0.5);
\draw[->] (-1.25,2)--(-1.25, 2.25)--(-1.75,2.25)--(-1.75,2);
\draw (-1.5,2.5) node{$-t_1$};
\draw[->] (1.75,2)--(1.75, 2.25)--(1.25,2.25)--(1.25,2.03);
\draw(1.5,2.5) node{$t_1$};
\draw[->] (-0.15,2)--(-0.15, 2.25)--(-0.75,2.25)--(-0.75,2.03);
\draw (-0.4,2.55) node{$-t_1^q$};
\draw[->] (0.75,2)--(0.75, 2.25)--(0.25,2.25)--(0.25,2.03);
\draw(0.5,2.55) node{$t_1^q$};
\draw[ ->] (0.15,2.0)--(0.15, 3)--(-0.9,3)--(-0.9,2.03);
\draw (-0.65,3.4) node{$-t_1^{q^2}$};
\draw[->] (0.85,2.0)--(0.85, 3.5)--(-0.25,3.5)--(-0.25,3);
\draw(0.5,3.9) node{$t_1^{q^2}$};
\draw(0,-2.5) node{$A.x_1(t_1)$};
\end{tikzpicture}
\qquad
\begin{tikzpicture}[scale=0.75]
\draw[step=0.5cm, gray, very thin](-2, -2)grid(2, 2);
\draw[very thick] (-2,0)--(2,0);
\draw[very thick] (0,-2)--(0,2);
\draw (-1.75,1.75) node{$\bullet$};
\draw (-1.25,1.25) node{$\bullet$};
\draw (-0.75,0.75) node{$\bullet$};
\draw (-0.25,0.25) node{$\bullet$};
\draw (0.25,-0.25) node{$\bullet$};
\draw (0.75,-0.75) node{$\bullet$};
\draw (1.25,-1.25) node{$\bullet$};
\draw (1.75,-1.75) node{$\bullet$};
\draw (1.75,1.75) node{$\bullet$};
\draw (1.25,1.25) node{$\bullet$};
\draw (0.75,0.75) node{$\bullet$};
\draw (0.25,0.25) node{$\bullet$};
\foreach \x in {-1.5, -1,...,1.0}
{
\draw [very thick] (\x,1.5)   rectangle +(0.5,0.5);
}
\draw [very thick] (-1.0,1.0) rectangle +(0.5,0.5);
\draw[->] (-0.75,2)--(-0.75, 2.25)--(-1.85,2.25)--(-1.85,2);
\draw (-1.6,2.5) node{$t_3$};
\draw[->] (1.9,2)--(1.9, 2.25)--(0.75,2.25)--(0.75,2.03);
\draw(1.8,2.5) node{$-t_3$};
\draw[->] (-0.25,2)--(-0.25, 2.55)--(-1.15,2.55)--(-1.15,2.25);
\draw (-0.75,2.9) node{$-t_3^q$};
\draw[->] (1.15,2.25)--(1.15, 2.55)--(0.25,2.55)--(0.25,2.03);
\draw(0.75,2.9) node{$t_3^q$};
\draw[->] (0.15,2.0)--(0.15, 3.3)--(-1.3,3.3)--(-1.3,2.25);
\draw (-0.75,3.7) node{$-t_3^{q^2}$};
\draw[->] (1.3,2.25)--(1.3, 3.55)--(-0.25,3.55)--(-0.25,3.3);
\draw(0.75,3.95) node{$t_3^{q^2}$};
\draw(0,-2.5) node{$A.x_3(t_3)$};
\end{tikzpicture}
\qquad
\begin{tikzpicture}[scale=0.75]
\draw[step=0.5cm, gray, very thin](-2, -2)grid(2, 2);
\draw[very thick] (-2,0)--(2,0);
\draw[very thick] (0,-2)--(0,2);
\draw (-1.75,1.75) node{$\bullet$};
\draw (-1.25,1.25) node{$\bullet$};
\draw (-0.75,0.75) node{$\bullet$};
\draw (-0.25,0.25) node{$\bullet$};
\draw (0.25,-0.25) node{$\bullet$};
\draw (0.75,-0.75) node{$\bullet$};
\draw (1.25,-1.25) node{$\bullet$};
\draw (1.75,-1.75) node{$\bullet$};
\draw (1.75,1.75) node{$\bullet$};
\draw (1.25,1.25) node{$\bullet$};
\draw (0.75,0.75) node{$\bullet$};
\draw (0.25,0.25) node{$\bullet$};
\foreach \x in {-1.5, -1,...,1.0}
{
\draw [very thick] (\x,1.5)   rectangle +(0.5,0.5);
}
\draw [very thick] (-1.0,1.0) rectangle +(0.5,0.5);
\draw[->] (-0.25,2)--(-0.25, 2.25)--(-1.75,2.25)--(-1.75,2.0);
\draw (-0.8,2.5) node{$-t_4$};
\draw[->] (1.75,2)--(1.75, 2.25)--(0.25,2.25)--(0.25,2.03);
\draw(1.0,2.5) node{$t_4$};
\draw[->] (0.75,2.25)--(0.75, 2.8)--(-1.25,2.8)--(-1.25,2.25);
\draw (-0.25,3.1) node{$-t_4^q$};
\draw[->] (1.25,2.25)--(1.25, 3.45)--(-0.75,3.45)--(-0.75,2.8);
\draw(0.6,3.75) node{$t_4^q$};
\draw[->] (0.25,3.45)--(0.25, 3.75)--(-1.9,3.75)--(-1.9,2);
\draw (-1,4.15) node{$-t_4^{q^2}$};
\draw[->] (1.85,2)--(1.85, 4.25)--(-0.25,4.25)--(-0.25,3.75);
\draw(1,4.65) node{$t_4^{q^2}$};
\draw(0,-2.5) node{$A.x_4(t_4)$};
\end{tikzpicture}
\end{align*}
\begin{align*}
\begin{tikzpicture}[scale=0.75]
\draw[step=0.5cm, gray, very thin](-2, -2)grid(2, 2);
\draw[very thick] (-2,0)--(2,0);
\draw[very thick] (0,-2)--(0,2);
\draw (-1.75,1.75) node{$\bullet$};
\draw (-1.25,1.25) node{$\bullet$};
\draw (-0.75,0.75) node{$\bullet$};
\draw (-0.25,0.25) node{$\bullet$};
\draw (0.25,-0.25) node{$\bullet$};
\draw (0.75,-0.75) node{$\bullet$};
\draw (1.25,-1.25) node{$\bullet$};
\draw (1.75,-1.75) node{$\bullet$};
\draw (1.75,1.75) node{$\bullet$};
\draw (1.25,1.25) node{$\bullet$};
\draw (0.75,0.75) node{$\bullet$};
\draw (0.25,0.25) node{$\bullet$};
\foreach \x in {-1.5, -1,...,1.0}
{
\draw [very thick] (\x,1.5)   rectangle +(0.5,0.5);
}
\draw [very thick] (-1.0,1.0) rectangle +(0.5,0.5);
\draw[->] (-0.75,2)--(-0.75, 2.25)--(-1.25,2.25)--(-1.25,2.03);
\draw (-1.0,2.55) node{$-t_2$};
\draw[->] (1.25,2)--(1.25, 2.25)--(0.75,2.25)--(0.75,2.03);
\draw(1.0,2.55) node{$t_2$};
\draw(0,-2.5) node{$A.x_2(t_2)$};
\end{tikzpicture}
\qquad
\begin{tikzpicture}[scale=0.75]
\draw[step=0.5cm, gray, very thin](-2, -2)grid(2, 2);
\draw[very thick] (-2,0)--(2,0);
\draw[very thick] (0,-2)--(0,2);
\draw (-1.75,1.75) node{$\bullet$};
\draw (-1.25,1.25) node{$\bullet$};
\draw (-0.75,0.75) node{$\bullet$};
\draw (-0.25,0.25) node{$\bullet$};
\draw (0.25,-0.25) node{$\bullet$};
\draw (0.75,-0.75) node{$\bullet$};
\draw (1.25,-1.25) node{$\bullet$};
\draw (1.75,-1.75) node{$\bullet$};
\draw (1.75,1.75) node{$\bullet$};
\draw (1.25,1.25) node{$\bullet$};
\draw (0.75,0.75) node{$\bullet$};
\draw (0.25,0.25) node{$\bullet$};
\foreach \x in {-1.5, -1,...,1.0}
{
\draw [very thick] (\x,1.5)   rectangle +(0.5,0.5);
}
\draw [very thick] (-1.0,1.0) rectangle +(0.5,0.5);
\draw[->] (0.75,2)--(0.75, 2.25)--(-1.75,2.25)--(-1.75,2);
\draw (-0.25,2.5) node{$-t_5$};
\draw[->] (1.75,2)--(1.75, 2.85)--(-0.75,2.85)--(-0.75,2.25);
\draw(0.5,3.1) node{$t_5$};
\draw(0,-2.5) node{$A.x_5(t_5)$};
\end{tikzpicture}
\qquad
\begin{tikzpicture}[scale=0.75]
\draw[step=0.5cm, gray, very thin](-2, -2)grid(2, 2);
\draw[very thick] (-2,0)--(2,0);
\draw[very thick] (0,-2)--(0,2);
\draw (-1.75,1.75) node{$\bullet$};
\draw (-1.25,1.25) node{$\bullet$};
\draw (-0.75,0.75) node{$\bullet$};
\draw (-0.25,0.25) node{$\bullet$};
\draw (0.25,-0.25) node{$\bullet$};
\draw (0.75,-0.75) node{$\bullet$};
\draw (1.25,-1.25) node{$\bullet$};
\draw (1.75,-1.75) node{$\bullet$};
\draw (1.75,1.75) node{$\bullet$};
\draw (1.25,1.25) node{$\bullet$};
\draw (0.75,0.75) node{$\bullet$};
\draw (0.25,0.25) node{$\bullet$};
\foreach \x in {-1.5, -1,...,1.0}
{
\draw [very thick] (\x,1.5)   rectangle +(0.5,0.5);
}
\draw [very thick] (-1.0,1.0) rectangle +(0.5,0.5);
\draw[->] (1.25,2)--(1.25, 2.25)--(-1.75,2.25)--(-1.75,2);
\draw (-0.25,2.5) node{$-t_6$};
\draw[->] (1.75,2)--(1.75, 2.85)--(-1.25,2.85)--(-1.25,2.25);
\draw(0.25,3.1) node{$t_6$};
\draw(0,-2.5) node{$A.x_6(t_6)$};
\end{tikzpicture}
\end{align*}
\end{Lemma}
\begin{Lemma}\label{zeta_u}
 If $u\in \mathbb{F}_{q^3}^*$ and $p>2$,
 then the map
$\zeta_u \colon \mathbb{F}_{q^3} \to  \mathbb{F}_{q^3}: t\mapsto ut^{q^2}+u^qt^q$
is an $\mathbb{F}_{q}$-automorphism.
\end{Lemma}
\begin{proof}
 c.f. Lemma 3.6 of \cite{sun2016arxiv}.
\end{proof}
\begin{Corollary}\label{zeta_u-b}
If $p>2$,
then
$
\mathbb{F}_{q^3} \to  \mathbb{F}_{q^3}: t\mapsto t+t^{q}
$
is an $\mathbb{F}_{q}$-automorphism.
\end{Corollary}

\begin{Lemma}[$U$-orbit modules]\label{prop: 3D4-orbit}
Let $A=(A_{i,j}) \in V$.
Then the $U$-orbit module $\mathbb{C}\mathcal{O}_U([A])$ $(A\in V)$ is obtained as follows:
\begin{align*}
& \mathbb{C}\mathcal{O}_U([A])\\
=& \mathbb{C} \Big\{
{\left[%
\newcommand{\mc}[3]{\multicolumn{#1}{#2}{#3}}
\begin{array}{cccccccc}
\cline{2-7}
\mc{1}{c|}{×}
& \mc{1}{c|}{
\begin{array}{l}
 A_{12}\\
-A_{13}t_2\\
-A_{15}^{q}t_3^q\\
-A_{16}t_1^{q^2}t_3^q\\
-A_{17}t_2t_1^{q^2}t_3^q\\
-A_{15}t_3^{q^2}\\
-A_{16}t_1^{q}t_3^{q^2}\\
-A_{17}t_2t_1^{q}t_3^{q^2}\\
-A_{17}t_3^{q^2+q}\\
-A_{16}t_4^q\\
-A_{17}t_2t_4^q\\
\end{array}
}
& \mc{1}{c|}{
\begin{array}{l}
A_{13}\\
-A_{15}^qt_1^q\\
-A_{15}t_1^{q^2}\\
-A_{16}t_1^{q^2+q}\\
-A_{17}t_2t_1^{q^2+q}\\
+A_{17}t_4^q\\
\end{array}
}
& \mc{1}{c|}{
\begin{array}{l}
A_{15}^{q}\\
+A_{16}t_1^{q^2}\\
+A_{17}t_2t_1^{q^2}\\
+A_{17}t_3^{q^2}\\
\end{array}
}
& \mc{1}{c|}{
\begin{array}{l}
A_{15}\\
+A_{16}t_1^{q}\\
+A_{17}t_2t_1^{q}\\
+A_{17}t_3^q\\
\end{array}}
& \mc{1}{c|}{
\begin{array}{l}
A_{16}\\
+A_{17}t_2\\
\end{array}
}
& \mc{1}{c|}{A_{17}}
& × \\\cline{2-7}
× & \mc{1}{c|}{×} & \mc{1}{c|}{A_{23}} & × & × & × & × & ×\\\cline{3-3}
  \end{array}
\right]} \\
& \phantom{\mathbb{C} \Big\{}
\Big|{\ }
t_1, t_3, t_4\in \mathbb{F}_{q^3},{\ } t_2 \in \mathbb{F}_{q}\Big\}.
\end{align*}
\end{Lemma}
\begin{proof}
 By \ref{3D4-A.xi, figures},
 we calculate the orbit modules directly.
\end{proof}

\begin{Definition}
 The elements of $V$  are called \textbf{patterns}.
 The monomial action of $G$ on $\mathbb{C}U$: $([A],g)\mapsto [A]*g$
 (e.g. \ref{fund thm U-3D4})
 and also the corresponding permutation operation on $V$: $(A,g)\mapsto A.g$
 (e.g. \ref{action A dot g-3D4})
 are called
 \textbf{truncated column operation}.
 So the permutation operation $(A,g)\mapsto A.g {\ }(A\in V, g\in G)$
 may be derived by a sequence of truncated column operations
 (see \ref{3D4-A.xi, figures})
 by \ref{pairwise diff-G8q3}.
\end{Definition}
Suppose $A\in V$.
Then
 $(i,j)\in J$ is a \textbf{main condition} of $A$
if and only if $A_{ij}$ is the rightmost non-zero entry in the $i$-th row.
We set
$\mathrm{main}(A):=
 \{ (i,j)\in J  \mid  (i,j) \text{ is a main condition of }A\}$.
 The coordinate $(i,j)$ is called \textbf{the $i$-th main condition} of $A$,
if $(i,j)\in \mathrm{main}(A)$. Set
$\mathrm{main}_i(A):= \{ (i,j)\in J \mid (i,j) \text{ is the $i$-th main condition of }A\}$.
Note that there exists at most one $i$-th main condition of $A$.

\begin{Definition}[Staircase pattern]\label{def: Staircase pattern-3D4}
Let $A\in V$ be a pattern.
We call $A$ a \textbf{staircase pattern},
if the elements in $\mathrm{main}(A)$ lie in different columns.
Analogously, a $U$-orbit
module $\mathbb{C}\mathcal{O}_U([A])$ is called a \textbf{staircase $U$-module},
if the elements in $\mathrm{main}(A)$ lie in different columns.
\end{Definition}

The \textbf{verge} of $A\in V$ is
$\mathrm{verge}(A):=
  \sum_{(i,j)\in \mathrm{main}(A)}{A_{i,j}e_{i,j}}$.
The \textbf{$i$-th verge} of $A$ is
$\mathrm{verge}_i(A):=
  \sum_{(i,k)\in \mathrm{main}_i (A)}{A_{i,k}e_{i,k}}$.
The (staircase) pattern $A\in V$ is called the \textbf{(staircase) verge pattern},
if $A=\mathrm{verge}(A)$.
A \textbf{minor condition}
 of $A\in V$ is $(i,j)\in J$ $(j\leq 4)$,
if $(i,9-j)$ is a main condition.
Set
$\mathrm{minor}(A):=
 \{(i,j)\in J  \mid (i,j) \text{ is a minor condition of }A\}\subseteq J$.
The \textbf{core}
of $A\in V$ is denoted by
$\mathrm{core}(A):=\mathrm{main}(A) \dot{\cup} \mathrm{minor}(A)$.
A (staircase) pattern $A\in V$ is a \textbf{(staircase) core pattern}
if $\mathrm{supp}(A) \subseteq \mathrm{core}(A)$.
\begin{Notation}
Define the families of $U$-orbit modules as follows:
$\mathfrak{F}_6:= \{\mathbb{C}\mathcal{O}_U(A) \mid  A\in V,{\,}A_{17}\neq 0\}$,
$\mathfrak{F}_5:= \{\mathbb{C}\mathcal{O}_U(A) \mid A\in V,{\,} A_{16}\neq 0,{\,} A_{17}= 0\}$,
$ \mathfrak{F}_4:= \{\mathbb{C}\mathcal{O}_U(A) \mid A\in V,{\,} A_{15}\neq 0,{\,} A_{16}=A_{17}= 0\}$,
$ \mathfrak{F}_3:= \{\mathbb{C}\mathcal{O}_U(A) \mid A\in V,{\,} A_{13}\neq 0,{\,} A_{15}=A_{16}=A_{17}= 0\}$,
$ \mathfrak{F}_{1,2}:= \{\mathbb{C}\mathcal{O}_U(A) \mid A\in V,{\,} A_{13}= A_{15}=A_{16}=A_{17}= 0\}$.
Let $A\in V$, we also say $A\in \mathfrak{F}_i$,
if $\mathbb{C}\mathcal{O}_U([A])\in \mathfrak{F}_i$.
\end{Notation}
\begin{Notation}\label{notation:Ta}
Let $a^*\in \mathbb{F}_{q^3}^*=\mathbb{F}_{q^3}\backslash \{0\}$,
and denote by
 $T^{a^*}$ a complete set of
 coset representatives (i.e. a transversal)
 of  $(a^*\mathbb{F}_{q}^+)$ in $\mathbb{F}_{q^3}^+$.
Then $|T^{a^*}|=q^2$.
If $\bar{t}_0\in T^{a^*}$ and $\bar{t}_0\in a^*\mathbb{F}_{q}^+$,
then we set $\bar{t}_0=0$.
\end{Notation}

\begin{Proposition}[Classification of $U$-orbit modules]\label{prop:class orbit-3D4}
Every $U$-orbit module is contained in one of the families
$\{\mathfrak{F}_{1,2},\mathfrak{F}_{3},\mathfrak{F}_{4}, \mathfrak{F}_{5},\mathfrak{F}_{6}\}$,
and
\begin{align*}
\mathfrak{F}_6=& \{\mathbb{C}\mathcal{O}_U([A_{12}e_{12}+A_{23}e_{23}+A_{17}^*e_{17}])
         \mid A_{12}\in \mathbb{F}_{q^3},A_{23}\in \mathbb{F}_{q}, A_{17}^*\in \mathbb{F}_{q}^*\},\\
\mathfrak{F}_5=& \{\mathbb{C}\mathcal{O}_U([A_{13}e_{13}+A_{23}e_{23}+A_{16}^*e_{16}])
         \mid A_{13}\in \mathbb{F}_{q^3},A_{23}\in \mathbb{F}_{q}, A_{6}^*\in \mathbb{F}_{q}^*\},\\
\mathfrak{F}_4=& \{\mathbb{C}\mathcal{O}_U([A_{23}e_{23}+{A_{15}^*}^qe_{14}+A_{15}^*e_{15}])
         \mid A_{23}\in \mathbb{F}_{q}, A_{15}^*\in \mathbb{F}_{q^3}^*\},\\
\mathfrak{F}_3=& \{\mathbb{C}\mathcal{O}_U([A_{23}e_{23}+\bar{A}_{12}^{A_{13}^*}e_{12}+A_{13}^*e_{13}])
         \mid A_{23}\in \mathbb{F}_{q}, A_{13}^*\in \mathbb{F}_{q^3}^*, \bar{A}_{12}^{A_{13}^*}\in T^{A_{13}^*}\},\\
\mathfrak{F}_{1,2}=& \{\mathbb{C}\mathcal{O}_U([A_{12}e_{12}+A_{23}e_{23}])
         \mid A_{12}\in \mathbb{F}_{q^3},A_{23}\in \mathbb{F}_{q}\}.
\end{align*}
The dimensions of $U$-orbit modules are determined.
In particular, every $U$-orbit module of
families $\mathfrak{F}_{1,2}$, $\mathfrak{F}_{4}$, $\mathfrak{F}_{5}$ and $\mathfrak{F}_{6}$
(i.e. except the family $\mathfrak{F}_{3}$)
contains one and only one staircase core pattern.
\end{Proposition}
\begin{proof}
Let $A=(A_{ij})\in V$ with $A_{17}=A_{17}^*\in \mathbb{F}_q^*$.
Then we have that
\begin{align*}
&\mathbb{C}\mathcal{O}_U([A])
=\mathbb{C} \left\{
{%
\left[
\newcommand{\mc}[3]{\multicolumn{#1}{#2}{#3}}
\begin{array}{cccccccc}
\cline{2-7}
\mc{1}{c|}{}
& \mc{1}{c|}{
\begin{array}{l}
 A_{12}\\
+\frac{A_{13}A_{16}+A_{15}^{q+1}}{A_{17}^*}\\
-\frac{B_{13}B_{16}+B_{15}^{q+1}}{A_{17}^*}
\end{array}
}
& \mc{1}{c|}{B_{13}}
& \mc{1}{c|}{B_{15}^{q^2}}
& \mc{1}{c|}{B_{15}}
& \mc{1}{c|}{B_{16}}
& \mc{1}{c|}{A^*_{17}}
& \\\cline{2-7}
× & \mc{1}{c|}{} & \mc{1}{c|}{A_{23}} & × & × & × &  & ×\\\cline{3-3}
  \end{array}
\right]}
{\, }\middle|{\,}
{\left\{\begin{array}{l}
     B_{13},B_{15}\in \mathbb{F}_{q^3}\\
     B_{16} \in \mathbb{F}_{q}
     \end{array}
     \right.}
\right\}
\end{align*}
Thus $\mathrm{dim}_{\mathbb{C}}\mathbb{C}\mathcal{O}_U([A])=q^7$.
Let $u:=x(t_1,
-\frac{A_{16}}{A_{17}^*},
-\frac{A_{15}^{q^2}}{A_{17}^*},
-\frac{A_{13}^{q^2}-(A_{15}t_1+A_{15}^{q^2}t_1^q)}{A_{17}^*},
t_5,t_6)$.
Then there is a staircase core pattern
\begin{align*}
C:=A.u=
 {\newcommand{\mc}[3]{\multicolumn{#1}{#2}{#3}}
\begin{array}{cccccccc}\cline{2-7}
\mc{1}{c|}{}
& \mc{1}{c|}{
\rule{0pt}{16pt}A_{12}+\frac{A_{13}A_{16}+A_{15}^{q+1}}{A_{17}^*}}
& \mc{1}{c|}{} & \mc{1}{c|}{}
                 & \mc{1}{c|}{} & \mc{1}{c|}{} & \mc{1}{c|}{A^*_{17}} & \\\cline{2-7}
× & \mc{1}{c|}{} & \mc{1}{c|}{A_{23}} & × & × & × &  & ×\\\cline{3-3}
  \end{array}}
  \in \mathcal{O}_U(A).
\end{align*}
Thus $\mathcal{O}_U(C)=\mathcal{O}_U(A)$
and $\mathbb{C}\mathcal{O}_U([A])=\mathbb{C}\mathcal{O}_U([C])$.
Since $C$ only depends on $A$,
the staircase core pattern is determined uniquely.
Thus
\begin{align*}
\mathfrak{F}_6=\{\mathbb{C}\mathcal{O}_U([D_{12}e_{12}+D_{23}e_{23}+D_{17}^*e_{17}])
         \mid D_{12}\in \mathbb{F}_{q^3},D_{23}\in \mathbb{F}_{q}, D_{17}^*\in \mathbb{F}_{q}^*\}.
\end{align*}
Similarly, all of the statements are proved.
\end{proof}

\begin{Remark}\label{3D4-staircase, F3}
Let $A\in V$.
In \ref{prop:class orbit-3D4},
if $\mathbb{C}\mathcal{O}_U([A]) \subseteq \mathfrak{F}_3$
and $A_{2,3}\neq 0$,
then $\mathbb{C}\mathcal{O}_U([A])$ is not a staircase orbit  module.
In all the other families, the orbit modules $\mathbb{C}\mathcal{O}_U([A])$
are staircase orbit modules.
\end{Remark}

\begin{Proposition}[$U$-stabilizer]\label{prop: 3D4-stab}
Let $A=(A_{ij})\in V$.
\begin{itemize}
\setlength\itemsep{0em}
 \item [(1)] If $A\in \mathfrak{F}_{1,2}$,
             then $\mathrm{Stab}_U(A)=U={{^3}D}^{syl}_4{(q^3)}$.
 \item [(2)] If  $A\in \mathfrak{F}_{3}$ and $A_{13}=A_{13}^*\in \mathbb{F}_{q^3}^*$,
             then $\mathrm{Stab}_U(A)=X_1X_3X_4X_5X_6$.
 \item [(3)] If  $A\in \mathfrak{F}_{4}$ and $A_{15}=A_{15}^*\in \mathbb{F}_{q^3}^*$, then
             \begin{align*}
 \mathrm{Stab}_U(A)
=&\left\{x(0,t_2,t_3,t_4,t_5,t_6){\, }
\middle| {\, }
            t_3, t_4\in \mathbb{F}_{q^3},
            t_2,  t_5, t_6 \in \mathbb{F}_q,
            {A^*_{15}}^qt_3^q+A^*_{15}t_3^{q^2}+A_{13}t_2=0
     \right\}\\
=&\bigg\{x(0,t_2,
-\frac{A_{13}^{q^2}{A_{15}^*}^{q}+A_{13}^q{A_{15}^*}-A_{13}{A_{15}^*}^{q^2}}{2{A_{15}^*}^{q+1}}t_2,
        t_4,t_5,t_6)
{\,}\bigg|{\,}
     t_4\in \mathbb{F}_{q^3},{\,}
     t_2,t_5, t_6 \in \mathbb{F}_q
     \bigg\}.
\end{align*}
 \item [(4)] If  $A\in \mathfrak{F}_{5}$ and $A_{16}=A_{16}^*\in \mathbb{F}_{q}^*$, then
 \begin{align*}
  \mathrm{Stab}_U(A)=\bigg\{
  x(0,t_2,t_3,\frac{-A_{13}^{q^2}t_2-A_{15}t_3-A_{15}^{q^2}t_3^q}{A^*_{16}},t_5,t_6)
 {\,}\bigg|{\,}
  t_3\in \mathbb{F}_{q^3},{\ }t_2,t_5,t_6\in \mathbb{F}_{q} \bigg\}.
 \end{align*}
 \item [(5)] If  $A\in \mathfrak{F}_{6}$ and $A_{17}=A_{17}^*\in \mathbb{F}_{q}^*$, then
 \begin{align*}
  \mathrm{Stab}_U(A)=\bigg\{
x(t_1,0,-\frac{A_{16}t_1}{A^*_{17}},\frac{A_{15}t_1+A_{15}^{q^2}t_1^q+A_{16}t_1^{q+1}}{A^*_{17}},t_5,t_6)
 {\,}\bigg|{\,}
t_1\in \mathbb{F}_{q^3},{\ }t_5,t_6\in \mathbb{F}_{q}\bigg\}.
 \end{align*}
\end{itemize}
\end{Proposition}
\begin{proof}
Let $A\in \mathfrak{F}_{4}$ (i.e. $A_{17}=A_{16}=0$, $A_{15}\neq 0$)
and $x:=x(t_1,t_2,t_3,t_4,t_5,t_6)\in U$.
Then we have that
\begin{align*}
& A.x=A
\stackrel{\ref{prop: 3D4-orbit}}{\iff}
\left\{
\begin{array}{l}
-{A^*_{15}}^qt_1^q-A^*_{15}t_1^{q^2}=0\\
-A_{13}t_2-{A^*_{15}}^{q}t_3^q-A^*_{15}t_3^{q^2}=0\\
\end{array}
\right.
\stackrel{\ref{zeta_u}}{\iff}
\left\{
\begin{array}{l}
t_1=0\\
A_{13}t_2+{A^*_{15}}^{q}t_3^q+A^*_{15}t_3^{q^2}=0\\
\end{array}
\right.
\\
{\iff}&
\left\{
\begin{array}{l}
t_1=0\\
A_{13}{A_{15}^*}^{q^2}t_2+{A_{15}^*}^{q^2+q}t_3^q+{A_{15}^*}^{q^2+1}t_3^{q^2}=0\\
A_{13}^q{A_{15}^*}t_2+{A_{15}^*}^{q^2+1}t_3^{q^2}+{A_{15}^*}^{q+1}t_3=0\\
A_{13}^{q^2}{A_{15}^*}^qt_2+{A_{15}^*}^{q+1}t_3+{A_{15}^*}^{q^2+q}t_3^{q}=0\\
\end{array}
\right.
{\iff}
\left\{
\begin{array}{l}
t_1=0\\
t_3=-\frac{A_{13}^{q^2}{A_{15}^*}^{q}+A_{13}^q{A_{15}^*}-A_{13}{A_{15}^*}^{q^2}}{2{A_{15}^*}^{q+1}}t_2
\end{array}
\right..
\end{align*}
Thus (4) is proved.
Similarly, all the other stabilizers are obtained.
\end{proof}

\begin{Comparison}
\label{com:classification and stabilizer-3D4}
\begin{itemize}
\setlength\itemsep{0em}
\item [(1)] (Classification of orbit modules).
Every (staircase) $A_n(q)$-orbit module
has precisely one (staircase) verge pattern (see \cite[Theorem 3.2]{yan2}).
Every (staircase) $D_n^{syl}(q)$-orbit module
has one and only one (staircase) core pattern (see \cite[3.2.29]{Markus1}).
Neither (staircase) verge patterns nor core patterns can do the classification
for (staircase) ${^3}D_4^{syl}(q^3)$-orbit modules
(e.g. the family $\mathfrak{F}_3$ of \ref{prop:class orbit-3D4}).
 \item [(2)] (Stabilizer).
Every (staircase) $A_n(q)$-orbit module has a basis element whose stabilizer is a pattern subgroup
(see \cite[\S 3.3]{yan2}).
This does not hold for
$D_n^{syl}(q)$-orbit modules (see \cite[3.2.25]{Markus1}),
or
for ${^3}D_4^{syl}(q^3)$-orbit modules
(e.g. the family $\mathfrak{F}_5$ of \ref{prop: 3D4-stab}).
\end{itemize}
\end{Comparison}


\section{Homomorphisms between orbit modules}
\label{sec:homom. between orbit modules-3D4}
Let $U:={^3}D_4^{syl}(q^3)$.
In this section,
we define a truncated row operation of $U$ on $V$ (\ref{def: row operation}).
Then we show that
every $U$-orbit module is isomorphic to a staircase orbit module (\ref{3D4-orbit to staircase}).
After that, some irreducible modules are determined,
and any two orbit modules are shown to be orthogonal
when the $1$st verges are different (\ref{3D4-orth.}).

The following property is well known:
every $\varphi \in {\mathrm{End}_{\mathbb{C}U}(\mathbb{C}U)}$ is of the form
$
 \lambda_a \colon  {\mathbb{C}U} \to  {\mathbb{C}U}: x\mapsto ax$
 for a unique $a\in {\mathbb{C}U}$.
 Let $g\in U$ and $A\in V$.
 Then
$
 \lambda_g|_{\mathbb{C}\mathcal{O}_U([A])} \colon
 {\mathbb{C}\mathcal{O}_U([A])} \to  \mathrm{Im ( \lambda_g|_{\mathbb{C}\mathcal{O}_U([A])}) }
 = g{\mathbb{C}\mathcal{O}_U([A])}
$
is a ${\mathbb{C}U}$-isomorphism.

\begin{Lemma}\label{3D4-g[A],lemma}
 Let $A\in V$ and $g\in U$.
 Then
$
 \lambda_g([ A ])
 {=}g{[ A ]}
={\frac{1}{|U|}} \sum_{y\in U}{\overline{\vartheta\kappa_q(g^{-\top}A, y)}}y$.
\end{Lemma}
\begin{proof}
If $g\in U$ and $A\in V$, then
 $\lambda_g([ A ])
 =\frac{1}{|U|}\sum_{z\in U}{\overline{\chi_A(z)}}gz
\stackrel{y:=gz}{=} \frac{1}{|U|}\sum_{y\in U}{\overline{\chi_A(g^{-1}y)}}y
                 =\frac{1}{|U|}\sum_{y\in U}{\overline{\vartheta\kappa_q(A, f(g^{-1}y))}}y$.
Since $f(g^{-1}y)\equiv g^{-1}y \mod V^\bot$, we have
\begin{align*}
\lambda_g([ A ])=\frac{1}{|U|}\sum_{y\in U}{\overline{\vartheta\kappa_q(A, g^{-1}y)}}y
=\frac{1}{|U|}\sum_{y\in U}{\overline{\vartheta\kappa_q(g^{-\top}A, y)}}y.
\end{align*}
\end{proof}

\begin{Lemma}\label{lemma:3D4-prop x.A}
Let $A=(A_{ij})\in V$ and $x:=x(t_1,t_2,t_3,t_4,t_5,t_6)\in U$.
Then
$ \pi(x^{-\top}A))=A-\pi_q(t_1A_{13})e_{23}$.
\end{Lemma}

\begin{Definition/Lemma}[${{^3}D_4^{syl}(q^3)}$-truncated row operation] \label{def: row operation}
The map
\begin{displaymath}
 U\times V  \to  V: (u,A)\mapsto u.A:=\pi(u^{-\top}A)
\end{displaymath}
defines a (left) group action, which is called the {\textbf{truncated row operation}}.
Note that the elements of $U$ act as $\mathbb{F}_q$-automorphisms on $V$.
\end{Definition/Lemma}
\begin{proof}
Since $\pi$ is an $\mathbb{F}_q$-linear map by \ref{pi,3D4},
it is enough to prove that $(gh).A=g.(h.A)$
for all $g:=x(t_1,t_2,t_3,t_4,t_5,t_6)\in U$,
$h:=x(s_1,s_2,s_3,s_4,s_5,s_6)\in U$ and $A=(A_{ij})\in V$.
Let $gh:=x(r_1,r_2,r_3,r_4,r_5,r_6)$.
Then $r_1=t_1+s_1$ by \ref{sylow p-subg, 3D4}.
By \ref{lemma:3D4-prop x.A},
we have
$g.(h.A)=g.(A-\pi_q(s_1A_{13})e_{23})
=A-(\pi_q(t_1A_{13})+\pi_q(s_1A_{13}))e_{23}
=A-\pi_q(r_1A_{13})e_{23}
=(gh).A$.
\end{proof}

\begin{Corollary}\label{3D4-prop x.A}
 Let $A\in V$ and $x:=x(t_1,t_2,t_3,t_4,t_5,t_6)\in U$.
 Then
 $ x.A=A-\pi_q(t_1A_{13})e_{23}$.
 In particular,
 $x_1(t_1).A=A-\pi_q(t_1A_{13})e_{23}$
 and $x_i(t_i).A=A$ for all $i\in \{2,3,4,5,6\}$.
\end{Corollary}
\begin{Remark}
Let $A\in V$ and $g,u\in U$.
In general $g.(A.u)\neq (g.A).u$.
For example,
if $t_1,t_4\in \mathbb{F}_{q^3}^*$ and
$A=A_{17}^*e_{17}$ ($A_{17}^*\in \mathbb{F}_{q}^*$),
then
$\left(x_1(t_1).A\right).x_4(t_4)=A+t_4^qA_{17}^*e_{13}$
and
$x_1(t_1).\left(A.x_4(t_4)\right)=A+t_4^qA_{17}^*e_{13}-t_1t_4^qA_{17}^*e_{23}$,
so $\left(x_1(t_1).A\right).x_4(t_4)\neq x_1(t_1).\left(A.x_4(t_4)\right)$.
\end{Remark}

\begin{Lemma}\label{3D4-g[B] obit, Lemma}
Let $B:=B_{12}e_{12}+B_{13}e_{13}+B_{23}e_{23}\in V$,
$g:=x(t_1,t_2,t_3,t_4,t_5,t_6)\in U$
and $y\in U$.
Then
$\vartheta\kappa_q(g^{-\top}B, y-1)=\chi_{g.B}(y)$.
In particular,
$\vartheta\kappa_q(x_1(t_1)^{-\top}B, y-1)=\chi_{x_1(t_1).B}(y)$
for all $t_1\in \mathbb{F}_{q^3}$.
\end{Lemma}
\begin{proof}
Let $B:=B_{12}e_{12}+B_{13}e_{13}+B_{23}e_{23}\in V$,
$g:=x(t_1,t_2,t_3,t_4,t_5,t_6)\in U$ and $y\in U$.
Then
\begin{align*}
& \vartheta\kappa_q(g^{-\top}B, y-1)
  {=}\vartheta \kappa_q(B-t_1B_{12}e_{22}-t_1B_{13}e_{23}, y-1)\\
{=}& \vartheta \big(\kappa_q(B, y-1)+\pi_q((-t_1B_{13})y_{23})\big)
{=} \vartheta \big(\kappa_q(B, y-1)+\pi_q(\pi_q(-t_1B_{13})y_{23})\big)\\
{=}& \vartheta\kappa_q(B-\pi_q(t_1B_{13})e_{23}, y-1)
\stackrel{\ref{3D4-prop x.A}}{=}
 \vartheta\kappa_q(g.B, f(y))
= \chi_{g.B}(y).
\end{align*}
\end{proof}

\begin{Proposition}\label{3D4-g[B] orbit,Prop}
Let $g\in U$ and $A:=A_{12}e_{12}+A_{13}e_{13}+A_{23}e_{23}\in V$.
Then
\begin{displaymath}
 \lambda_g([B])=\chi_{g.B}(g)[g.B] \quad \text{ for all } B\in \mathcal{O}_U(A).
\end{displaymath}
\end{Proposition}
\begin{proof}
If $B\in \mathcal{O}_U(A)$,
then $\mathrm{main}(B)=\mathrm{main}(A)$.
Thus
\begin{align*}
g[B]
\stackrel{\ref{3D4-g[A],lemma}}{=}
& {\frac{1}{|U|}} \sum_{y\in U}{\overline{\vartheta\kappa_q(g^{-\top}B, y)}}y
= {\frac{1}{|U|}} \sum_{y\in U}
  {\overline{\vartheta\kappa_q(g^{-\top}B, y-1)\vartheta\kappa_q(g^{-\top}B, 1)}}y\\
=& {\frac{1}{|U|}} \sum_{y\in U}
  {\overline{\vartheta\kappa_q(g^{-\top}B, y-1)\vartheta\kappa_q(B, f(g^{-1}))}}y\\
\stackrel{\ref{3D4-g[B] obit, Lemma}}{=}&
  {\overline{\vartheta\kappa_q(B, f(g^{-1}))}}\cdot{\frac{1}{|U|}} \sum_{y\in U}
  {\overline{\chi_{g.B}(y)}}y
{=} \overline{\chi_B(g^{-1})}[g.B].
\end{align*}
We have
$\chi_{g.B}(g)
\stackrel{\ref{3D4-g[B] obit, Lemma}}{=} \vartheta\kappa_q(g^{-\top}B, g-1)
= \vartheta\kappa_q(B, 1-g^{-1})
=\overline{\vartheta\kappa_q(B, g^{-1}-1)}
=\overline{\vartheta\kappa_q(B, f(g^{-1}))}
= \overline{\chi_B(g^{-1})}$.
Hence $\lambda_g([B])=\chi_{g.B}(g)[g.B]$ for all $B\in \mathcal{O}_U(A)$.
\end{proof}

\begin{Corollary}\label{right action A23 image-2D4}
Let $g\in U$ and $A:=A_{12}e_{12}+A_{13}e_{13}+A_{23}e_{23}\in V$.
Then
$\mathrm{im}(\lambda_g|_{\mathbb{C}\mathcal{O}_{U}([A])})
=\mathbb{C}\mathcal{O}_{U}([g.A])$,
and
$g.(B.u)= (g.B).u$ for all  $B\in \mathcal{O}_U(A)$ and $u\in U$.
\end{Corollary}
\begin{proof}
If $x\in U$, then
$\mathbb{C}g[A.x]=\mathbb{C}g([A]x)
=\mathbb{C}(g[A])x
\stackrel{\ref{3D4-g[B] orbit,Prop}}{=}\mathbb{C}[g.A]x$.
So $\mathrm{im}(\lambda_g|_{\mathbb{C}\mathcal{O}_{U}([A])})=\mathbb{C}\mathcal{O}_{U}([g.A])$.
We have
$
\mathbb{C}[(g.B).u]
= \mathbb{C}[g.B] u
\stackrel{\ref{3D4-g[B] orbit,Prop}}{=} \mathbb{C}(g[B])u
{=} \mathbb{C}g([B]u)
{=} \mathbb{C}g[B.u]
{=} \mathbb{C}[g.(B.u)]$,
thus $ g.(B.u)=(g.B).u $.
\end{proof}

\begin{Corollary}\label{3D4-delete A23}
Let $A:=A_{12}e_{12}+A_{23}e_{23}+A_{13}^*e_{13}\in V$,
$A_{13}^*\in \mathbb{F}_{q^3}^*$, and $x_1(t_1)\in U$ such that $\pi_q(t_1A_{13}^*)=A_{23}$.
Then
$\mathbb{C}\mathcal{O}_{U}([A])   \cong  \mathbb{C}\mathcal{O}_{U}([x_1(t_1).A])
= \mathbb{C}\mathcal{O}_{U}([A-A_{23}e_{23}])$,
i.e.
\begin{align*}
\mathbb{C}\mathcal{O}_{U}(
{\left[%
\newcommand{\mc}[3]{\multicolumn{#1}{#2}{#3}}
\begin{array}{cccccccc}\cline{2-7}
\mc{1}{c|}{} & \mc{1}{c|}{A_{12}} & \mc{1}{c|}{A_{13}^*} & \mc{1}{c|}{×} & \mc{1}{c|}{×} & \mc{1}{c|}{×} & \mc{1}{c|}{×} & \\\cline{2-7}
× & \mc{1}{c|}{} & \mc{1}{c|}{A_{23}} & × & × & × &  & ×\\\cline{3-3}
     \end{array}
\right]}%
)
& \cong  \mathbb{C}\mathcal{O}_{U}(
{\left[%
\newcommand{\mc}[3]{\multicolumn{#1}{#2}{#3}}
\begin{array}{cccccccc}\cline{2-7}
\mc{1}{c|}{} & \mc{1}{c|}{A_{12}} & \mc{1}{c|}{A_{13}^*} & \mc{1}{c|}{×} & \mc{1}{c|}{×} & \mc{1}{c|}{×} & \mc{1}{c|}{×} & \\\cline{2-7}
× & \mc{1}{c|}{} & \mc{1}{c|}{0} & × & × & × &  & ×\\\cline{3-3}
     \end{array}
\right]}%
).
\end{align*}
\end{Corollary}

\begin{Corollary}\label{3D4-orbit to staircase}
Every $U$-orbit module is isomorphic to a (not necessarily unique) staircase module,
and the isomorphism is given by the left multiplication by a group element.
\end{Corollary}
\begin{proof}
Let $A\in V$ and $(1,3), (2,3)\in \mathrm{main}(A)$.
Then $(2,3)$ is deleted by \ref{3D4-delete A23}.
By \ref{3D4-staircase, F3}, the claim is proved.
\end{proof}

\begin{Lemma}\label{3D4-x5 A17}
Let $A\in V$ with $A_{17}=A_{17}^* \in \mathbb{F}_q^*$,
$x_5(s_5)\in U$ and $s_5\in \mathbb{F}_q$.
Then
\begin{align*}
\lambda_{x_5(s_5)}([A])=\vartheta(s_5A_{16})[A+s_5{A_{17}^*}e_{23}].
\end{align*}
\end{Lemma}
\begin{proof}
Let $A\in V$ with $A_{17}=A_{17}^* \in \mathbb{F}_q^*$
and $x_5(s_5)\in U$.
Then we have that
\begin{align*}
&\lambda_{x_5(s_5)}([A])
\stackrel{\ref{3D4-g[A],lemma}}{=}
{\frac{1}{|U|}} \sum_{y\in U}{\overline{\vartheta\kappa_q(x_5(s_5)^{-\top}A, y)}}y
{=}
{\frac{1}{|U|}} \sum_{y\in U}{\overline{\vartheta\kappa_q(A-s_5A_{16}e_{66}-s_5{A_{17}^*}e_{67}, y)}}y\\
{=}&
\vartheta(s_5A_{16})\cdot
{\frac{1}{|U|}} \sum_{y\in U}{\big(\overline{\vartheta\kappa_q(A, y)}
\cdot \overline{\vartheta\pi_q(-s_5{A_{17}^*}y_{67})}\big)}y\\
\stackrel{\ref{sylow p-subg, 3D4}}{=}&
\vartheta(s_5A_{16})\cdot
{\frac{1}{|U|}} \sum_{y\in U}{\big(\overline{\vartheta\kappa_q(A, y)}
\cdot \overline{\vartheta\pi_q(s_5{A_{17}^*}y_{23})}\big)}y
{=}
\vartheta(s_5A_{16})\cdot
{\frac{1}{|U|}} \sum_{y\in U}{\overline{\vartheta\kappa_q(A+s_5{A_{17}^*}e_{23}, y)}}y\\
{=}&
\vartheta(s_5A_{16})\cdot
{\frac{1}{|U|}} \sum_{y\in U}{\overline{\vartheta\kappa_q(A+s_5{A_{17}^*}e_{23}, f(y))}}y
{=}
\vartheta(s_5A_{16}) [A+s_5{A_{17}^*}e_{23}].
\end{align*}
\end{proof}

\begin{Proposition}\label{3D4-delete A23, A17}
Let $A,B\in V$, $A_{17}=A_{17}^* \in \mathbb{F}_q^*$, and
\begin{align*}
A:= {%
\newcommand{\mc}[3]{\multicolumn{#1}{#2}{#3}}
\begin{array}{cccccccc}\cline{2-7}
\mc{1}{c|}{} & \mc{1}{c|}{A_{12}} & \mc{1}{c|}{A_{13}} & \mc{1}{c|}{A_{15}^q}
& \mc{1}{c|}{A_{15}} & \mc{1}{c|}{A_{16}} & \mc{1}{c|}{A_{17}^*} & \\\cline{2-7}
× & \mc{1}{c|}{} & \mc{1}{c|}{A_{23}} & × & × & × &  & ×\\\cline{3-3}
     \end{array}
},
\quad
B:=
{%
\newcommand{\mc}[3]{\multicolumn{#1}{#2}{#3}}
\begin{array}{cccccccc}\cline{2-7}
\mc{1}{c|}{} & \mc{1}{c|}{A_{12}} & \mc{1}{c|}{A_{13}} & \mc{1}{c|}{A_{15}^q}
& \mc{1}{c|}{A_{15}} & \mc{1}{c|}{A_{16}} & \mc{1}{c|}{A_{17}^*} & \\\cline{2-7}
× & \mc{1}{c|}{} & \mc{1}{c|}{0} & × & × & × &  & ×\\\cline{3-3}
     \end{array}
}.
\end{align*}
Then
$
\mathbb{C}\mathcal{O}_{U}([A])
\cong  \mathbb{C}\mathcal{O}_{U}([B])
$.
\end{Proposition}
\begin{proof}
Let $C\in \mathcal{O}_U(A)$ and $s_5:=-\frac{A_{23}}{A_{17}^*}\in \mathbb{F}_q$.
By \ref{3D4-x5 A17}, we get
$
\lambda_{x_5(s_5)}([C])=\vartheta(s_5C_{16})[C+s_5{A_{17}^*}e_{23}]
$
where $C+s_5{A_{17}^*}e_{23}\in \mathcal{O}_U(B)$ (c.f. the proof of \ref{right action A23 image-2D4}).
Thus $\mathbb{C}\mathcal{O}_{U}([A])
\cong  \mathbb{C}\mathcal{O}_{U}([B])$.
\end{proof}
For $1\leq i \leq 8$, the $i$\textbf{-th hook} of $J$ is
$H_i:=\{(a,b) \in J  \mid  b=i \text{ or } a=9-i\}$.
We have $H_7=\{(1,7), (2,3)\}$.
A pattern $A\in V$ is called
\textbf{hook-separated},
if on every hook $H_i$ of $J$ lies at most one main condition of $A$.
Note that hook-separated patterns are always staircase patterns.
If $A\in V$ is hook-separated,
then $\mathbb{C}\mathcal{O}_U([A])$ is called
a \textbf{hook-separated staircase module}.

\begin{Corollary}\label{3D4-classification, hook-sep.}
 Every $U$-orbit module
 is isomorphic to a certain hook-separated staircase module.
\end{Corollary}
\begin{proof}
 By \ref{3D4-orbit to staircase}, every $U$-orbit module is isomorphic to a
 staircase module.
 By \ref{3D4-delete A23, A17}, we get the desired conclusion.
\end{proof}
Let $A, B \in V$,
$\mathrm{Stab}_U(A, B):=\mathrm{Stab}_U(A)\cap \mathrm{Stab}_U(B)$,
$\psi_A$ be the character of $\mathbb{C}\mathcal{O}_U([A])$
and $\psi_B$ denote the character of $\mathbb{C}\mathcal{O}_U([B])$.
Then
$\mathrm{Hom}_{\mathbb{C}U}(\mathbb{C}\mathcal{O}_U([A]),\mathbb{C}\mathcal{O}_U([B]))=\{0\}$
if and only if for all $C\in \mathcal{O}_U(A)$ and $D\in \mathcal{O}_U(B)$ holds
$\mathrm{Hom}_{\mathrm{Stab}_U(C, D)}(\mathbb{C}[C],\mathbb{C}[D])=\{0\}$.
In particular,
\begin{align*}
&\mathrm{dim}_{\mathbb{C}} \mathrm{Hom}_{\mathbb{C}U}(\mathbb{C}\mathcal{O}_U([A]),\mathbb{C}\mathcal{O}_U([B]))
    = \langle \psi_A, \psi_B\rangle_{U}\\
   = &\sum_{\substack{C\in \mathcal{O}_U(A)\\ D\in \mathcal{O}_U(B)}}
   \frac{|\mathrm{Stab}_U(C, D)|}{|U|} \Big(\mathrm{dim}_{\mathbb{C}}
                          \mathrm{Hom}_{\mathrm{Stab}_U(C, D)}(\mathbb{C}[C],\mathbb{C}[D])\Big).
\end{align*}
If $y\in U$, then
$
 \mathrm{Hom}_{\mathrm{Stab}_U(A, B)}(\mathbb{C}[ A ],\mathbb{C}[ B ])=
 \mathrm{Hom}_{\mathrm{Stab}_U(A.y, B.y)}(\mathbb{C}[ A.y ],\mathbb{C}[ B.y ])$
as $\mathbb{C}$-vector spaces,
and
$\mathrm{Hom}_{\mathbb{C}U}(\mathbb{C}\mathcal{O}_U([A]),\mathbb{C}\mathcal{O}_U([B]))=\{0\}$
if and only if
$\mathrm{Hom}_{\mathrm{Stab}_U(A, D)}(\mathbb{C}[ A ],\mathbb{C}[D])=\{0\}$
for all $D\in \mathcal{O}_U(B)$ (\cite[\S 3.3]{Markus1}).
\begin{Corollary}
Let $A, B\in V$.
Then
$\langle \psi_A, \psi_B\rangle_{U}
=\sum_{D\in \mathcal{O}_U(B)}
    \frac{|\mathrm{Stab}_U(A, D)|}
    {|\mathrm{Stab}_U(A)|}{\langle \chi_A, \chi_D\rangle_{\mathrm{Stab}_U(A, D)}}$,
where
$\chi_A$ and $\chi_D$ are the characters of the $\mathbb{C}\mathrm{Stab}_U(A, D)$-modules
$\mathbb{C}[A]$ and $\mathbb{C}[D]$ respectively.
\end{Corollary}

\begin{Proposition}\label{3D4-orth.}
Every $U$-orbit module is isomorphic to a hook-separated staircase module
in Table \ref{Table: hook sep.-3D4},
and they satisfy the following properties.
\begin{itemize}
 \setlength\itemsep{0em}
\item [(1)]
Let $A,B\in V$.
If $\mathrm{verge}_1(A)\neq \mathrm{verge}_1(B)$,
then $\mathrm{Hom}_{\mathbb{C}U}(\mathbb{C}\mathcal{O}_U([A]),\mathbb{C}\mathcal{O}_U([B]))=\{0\}$.
In particular, if $\mathbb{C}\mathcal{O}_U([A])\in \mathfrak{F}_i$,
$\mathbb{C}\mathcal{O}_U([B])\in \mathfrak{F}_j$ and $i\neq j$,
then
$\mathrm{Hom}_{\mathbb{C}U}(\mathbb{C}\mathcal{O}_U([A]),\mathbb{C}\mathcal{O}_U([B]))=\{0\}$.
\item [(2)]
In the family $\mathfrak{F}_{1,2}$, the $q^4$ hook-separated staircase modules
are irreducible
and pairwise orthogonal.
\item [(3)]
In the family $\mathfrak{F}_3$, the $(q^3-1)q^2$ hook-separated staircase modules
are irreducible
and pairwise orthogonal.
\item [(4)]
In the family $\mathfrak{F}_{4}$, $\mathfrak{F}_{5}$ and $\mathfrak{F}_{6}$,
the hook-separated staircase modules
are reducible.
\end{itemize}
\begin{table}[!htp]
\caption{Hook-separated staircase ${{^3}D}^{syl}_4{(q^3)}$-orbit modules}
\label{Table: hook sep.-3D4}
\begin{align*}
\begin{array}{|c|l|c|c|}\hline
\text{Family}
& \multicolumn{1}{c|}{\mathbb{C}\mathcal{O}_U([A]){\ } (A\in V)}
& \mathrm{dim}_{\mathbb{C}}\mathbb{C}\mathcal{O}_U([A])
& \text{Irreducible}\\\hline
\hline
\mathfrak{F}_6
&\rule{0pt}{20pt}
\mathbb{C}\mathcal{O}_U\Big(
{\left[%
\newcommand{\mc}[3]{\multicolumn{#1}{#2}{#3}}
\begin{array}{cccccccc}\cline{2-7}
\mc{1}{c|}{} & \mc{1}{c|}{A_{12}} & \mc{1}{c|}{×} & \mc{1}{c|}{×} & \mc{1}{c|}{×} & \mc{1}{c|}{×} & \mc{1}{c|}{A_{17}^*} & \\\cline{2-7}
× & \mc{1}{c|}{} & \mc{1}{c|}{0} & × & × & × &  & ×\\\cline{3-3}
\end{array}
\right]}\Big)
  & q^{7}
  & \text{NO}\\[10pt]\hline
\mathfrak{F}_5
& \rule{0pt}{20pt}
\mathbb{C}\mathcal{O}_U\Big(
{\left[%
\newcommand{\mc}[3]{\multicolumn{#1}{#2}{#3}}
\begin{array}{cccccccc}\cline{2-7}
\mc{1}{c|}{} & \mc{1}{c|}{×} & \mc{1}{c|}{A_{13}} & \mc{1}{c|}{×} & \mc{1}{c|}{×} & \mc{1}{c|}{A_{16}^*} & \mc{1}{c|}{×} & \\\cline{2-7}
× & \mc{1}{c|}{} & \mc{1}{c|}{A_{23}} & × & × & × &  & ×\\\cline{3-3}
\end{array}
\right]}\Big)
& q^{6}
& \text{NO}\\[10pt]\hline
\mathfrak{F}_4
& \rule{0pt}{20pt}
\mathbb{C}\mathcal{O}_U\Big(
{\left[%
\newcommand{\mc}[3]{\multicolumn{#1}{#2}{#3}}
\begin{array}{cccccccc}\cline{2-7}
\mc{1}{c|}{} & \mc{1}{c|}{×} & \mc{1}{c|}{×}
& \mc{1}{c|}{{A_{15}^*}^q} & \mc{1}{c|}{A_{15}^*} & \mc{1}{c|}{×} & \mc{1}{c|}{×} & \\\cline{2-7}
× & \mc{1}{c|}{} & \mc{1}{c|}{A_{23}} & × & × & × &  & ×\\\cline{3-3}
\end{array}
\right]}\Big)
 & q^6
 & \text{NO} \\[10pt]\hline
\mathfrak{F}_3
&\rule{0pt}{25pt}
\mathbb{C}\mathcal{O}_U\Big(
{\left[%
\newcommand{\mc}[3]{\multicolumn{#1}{#2}{#3}}
\begin{array}{cccccccc}\cline{2-7}
\mc{1}{c|}{}
 & \mc{1}{c|}{\bar{A}_{12}^{A_{13}^*}} & \mc{1}{c|}{A_{13}^*}
& \mc{1}{c|}{×} & \mc{1}{c|}{×} & \mc{1}{c|}{×} & \mc{1}{c|}{×} & \\\cline{2-7}
× & \mc{1}{c|}{} & \mc{1}{c|}{0} & × & × & × &  & ×\\\cline{3-3}
\end{array}
\right]}\Big)
& q
& \text{YES}  \\[10pt]\hline
\mathfrak{F}_{1,2}
& \rule{0pt}{20pt}
\mathbb{C}\mathcal{O}_U\Big(
{\left[%
\newcommand{\mc}[3]{\multicolumn{#1}{#2}{#3}}
\begin{array}{cccccccc}\cline{2-7}
\mc{1}{c|}{} & \mc{1}{c|}{A_{12}} & \mc{1}{c|}{×} & \mc{1}{c|}{×} & \mc{1}{c|}{×} & \mc{1}{c|}{×} & \mc{1}{c|}{×} & \\\cline{2-7}
× & \mc{1}{c|}{} & \mc{1}{c|}{A_{23}} & × & × & × &  & ×\\\cline{3-3}
\end{array}
\right]}\Big)
 & 1
 & \text{YES} \\[10pt]\hline
 \end{array}
\end{align*}
where $A_{12}^*,A_{13}^*, A_{15}^* \in \mathbb{F}_{q^3}^*$,
$A_{16}^*,A_{17}^*, A_{23}^* \in \mathbb{F}_{q}^*$
and $\bar{A}_{12}^{A_{13}^*}\in T^{A_{13}^*}$.
\end{table}
\end{Proposition}

\begin{proof}
By \ref{3D4-orbit to staircase}
and \ref{3D4-classification, hook-sep.},
every $U$-orbit module is isomorphic to a hook-separated staircase module in Table \ref{Table: hook sep.-3D4}.
Let $A=A_{15}^*e_{15}+(A_{15}^*)^qe_{14}+A_{23}e_{23}\in \mathfrak{F}_4$,
$B=B_{15}^*e_{15}+(B_{15}^*)^qe_{14}+B_{23}e_{23}\in \mathfrak{F}_4$
(i.e. $A_{15}^*, B_{15}^*\in \mathbb{F}_{q^3}^*$)
and $C\in \mathcal{O}_U(B)$.
By \ref{prop: 3D4-stab}, we get
$
\mathrm{Stab}_U(A)
 \stackrel{A_{13}=0}{=}X_2X_4X_5X_6$.
Then
$\mathrm{Stab}_U(A, C)
=\left\{
\begin{array}{ll}
 X_2X_4X_5X_6, & C_{13}=0\\
 X_4X_5X_6, & C_{13}\neq 0\\
\end{array}
\right.$.
We calculate the inner product:
\begin{align*}
& \langle \chi_A, \chi_C\rangle_{\mathrm{Stab}_U(A, C)}
=\frac{1}{|{\mathrm{Stab}_U(A, C)}|}\sum_{y\in {\mathrm{Stab}_U(A, C)}} \vartheta \kappa_q(A-C, f(y))\\
=&\frac{1}{|{\mathrm{Stab}_U(A, C)}|}
  \sum_{y\in {\mathrm{Stab}_U(A, C)}}   \vartheta \kappa_q
  \left(
{%
\newcommand{\mc}[3]{\multicolumn{#1}{#2}{#3}}
\begin{array}{c|c|cccc}\hline
\mc{1}{|c|}{-C_{12}} & -C_{13} & \mc{1}{c|}{(A_{15}^*-B_{15}^*)^q}
                               & \mc{1}{c|}{A_{15}^*-B_{15}^*} & \mc{1}{c|}{×} & \mc{1}{c|}{×}\\\hline
× & A_{23}-B_{23} & × & × & × & ×\\\cline{2-2}
\end{array}
}%
,
f(y)\right).
\end{align*}
{If $C_{13}=0$, then }
\begin{align*}
&0\neq  \mathrm{dim}_\mathbb{C} \mathrm{Hom}_{\mathrm{Stab}_U(A, C)}(\mathbb{C}[ A ],\mathbb{C}[ C ])
= \langle \chi_A, \chi_C\rangle_{\mathrm{Stab}_U(A, C)}\\
&=\frac{1}{|X_2X_4X_5X_6|}\sum_{\substack{t_4\in \mathbb{F}_{q^3}\\ t_2, t_5,t_6\in \mathbb{F}_{q}} }
        \bigg( \vartheta \pi_q\Big((A_{23}-B_{23})t_2
         +(A_{15}^*-B_{15}^*)^qt_4+(A_{15}^*-B_{15}^*)t_4^{q^2}\Big)\bigg)\\
&=\bigg(\frac{1}{q}\sum_{t_2\in \mathbb{F}_{q}}  \vartheta ((A_{23}-B_{23})t_2)\bigg)
\bigg(\frac{1}{q^3}\sum_{ t_4\in \mathbb{F}_{q^3}} \vartheta \phi_0\Big((A_{15}^*-B_{15}^*)t_4^{q^2}(\eta+\eta^q)\Big)\bigg)\\
\iff &  \{B_{23}= A_{23}\} \wedge \{B^*_{15}=A^*_{15}\}.
\end{align*}
{If $C_{13}\neq 0$, then}
\begin{align*}
&0\neq  \mathrm{dim}_\mathbb{C} \mathrm{Hom}_{\mathrm{Stab}_U(A, C)}(\mathbb{C}[ A ],\mathbb{C}[ C ])
  = \langle \chi_A, \chi_C\rangle_{\mathrm{Stab}_U(A, C)}\\
&=\frac{1}{|X_4X_5X_6|}\sum_{\substack{t_4\in \mathbb{F}_{q^3}\\  t_5,t_6\in \mathbb{F}_{q}} }
        \bigg( \vartheta \pi_q\Big((A_{15}^*-B_{15}^*)^qt_4+(A_{15}^*-B_{15}^*)t_4^{q^2}\Big)\bigg)\\
&=\frac{1}{q^3}\sum_{ t_4\in \mathbb{F}_{q^3}} \vartheta \phi_0\Big((A_{15}^*-B_{15}^*)t_4^{q^2}(\eta+\eta^q)\Big)\\
\iff &   B^*_{15}=A^*_{15}.
\end{align*}
{We get}
$
 \mathrm{Hom}_{\mathrm{Stab}_U(A, C)}(\mathbb{C}[ A ],\mathbb{C}[ C ])\neq\{0\}
\iff \langle \chi_A, \chi_C\rangle_{\mathrm{Stab}_U(A, C)}\neq 0 \quad (i.e. =1)
\iff   B^*_{15}=A^*_{15}$.
{Thus}
$
 \mathrm{Hom}_{\mathbb{C}U}(\mathbb{C}\mathcal{O}_U([A]),\mathbb{C}\mathcal{O}_U([B]))=\{0\}
\iff  B^*_{15}\neq A^*_{15}$.

Let $A\in \mathfrak{F}_i$ and  $B\in \mathfrak{F}_j$,
$\psi_A$ denote the character of $\mathbb{C}\mathcal{O}_U([A])$
and  $\psi_B$ denote the character of $\mathbb{C}\mathcal{O}_U([B])$.
In a similar way, we calculate $ \langle \psi_A, \psi_B\rangle_{U}$.
Then the statements of (1) are proved.

The $q^4$ hook-separated staircase modules of $\mathfrak{F}_{1,2}$ are of dimension $1$,
so they are irreducible,
and pairwise orthogonal by calculating inner product.

Let $A,B\in V$ be hook-separated staircase patterns of the family $\mathfrak{F}_3$ and $A\neq B$.
Then  $\langle \psi_A, \psi_A\rangle_{U}=1$ and $\langle \psi_A, \psi_B\rangle_{U}=0$,
thus the statements of (3) are proved.

Let $A\in V$ be a hook-separated staircase core pattern of the family $\mathfrak{F}_4$.
Then the orbit module $\mathbb{C}\mathcal{O}_U([A])$ is reducible.
Suppose that it is irreducible.
Then by (1) and (2) we get
$\left(\dim_{\mathbb{C}} \mathbb{C}\mathcal{O}_U([A])\right)^2=q^{12}< |U|-q^4=q^{12}-q^4$.
This is a contradiction.
Thus the orbit modules of the family $\mathfrak{F}_4$ are reducible.
Similarly, (5) and (6) are proved.
\end{proof}

\begin{Remark}\label{hook-separated intersect-3D4}
There exist two hook-separated staircase modules such that
they are neither orthogonal nor isomorphic.
For example,
if $A,B\in V$ are hook-separated staircase core patterns of family $\mathfrak{F}_4$
with $A_{15}^*=B_{15}^*\in \mathbb{F}_{q^3}^*$ and $A_{23}\neq B_{23}$,
then
$ \langle \psi_A, \psi_A\rangle_{U}
 =\langle \psi_B, \psi_B\rangle_{U}=q^5+q^3-q^2  $
but
$ \langle \psi_A, \psi_B\rangle_{U}=q^5-q^2 \notin\{0,{\,} q^5+q^3-q^2\}$,
so $\mathbb{C}\mathcal{O}_U([A])$ and $\mathbb{C}\mathcal{O}_U([B])$
are neither orthogonal nor isomorphic.
\end{Remark}

\begin{Comparison}
\label{com:classification staircase U-modules-3D4}
\begin{itemize}
\setlength\itemsep{0em}
\item [(1)] (Classification of staircase $U$-modules).
Let $U$ be $A_n(q)$, $D_n^{syl}(q)$ or ${^3}D_4^{syl}(q^3)$.
Then every $U$-orbit module is isomorphic to a staircase $U$-module
(see \cite[Prop. 2.2 and Thm. 3.2]{yan2},
\cite[3.3.15]{Markus1}
and \ref{3D4-orbit to staircase}).
\item[(2)] (Irreducible $U$-modules).
Every irreducible $A_n(q)$-module is a constituent of precisely one staircase module
(see \cite[Thm. 2.4 and Cor. 2.7]{yan2}).
Every irreducible $D_n^{syl}(q)$-module is a constituent of an unique hook-separated staircase module
(see \cite[3.3.19 and 3.3.43]{Markus1}).
Every irreducible ${^3D}_4^{syl}(q^3)$-module is a constituent of a
(not necessarily unique)
hook-separated staircase module
(see \ref{3D4-classification, hook-sep.},
\ref{3D4-orth.}
and \ref{hook-separated intersect-3D4}).
\end{itemize}
\end{Comparison}



\section{A partition of ${^3D}_4^{syl}(q^3)$}
\label{partition of U-3D4}
Let $G:=G_8(q^3)$ and $U:={^3}D_4^{syl}(q^3)$.
In this section,
a partition of ${^3D}_4^{syl}(q^3)$ is determined (see \ref{superclass:ci ti-3D4})
which is a set of superclasses proved in the next section \ref{sec: supercharacter theories-3D4}.

\begin{Lemma}
Let $1$ denote $I_8\in G$.
Then
$ V_G:=G-1=\{ g-1\mid g\in G \} $
is a nilpotent associative $\mathbb{F}_q$-algebra (G is an algebra group).
\end{Lemma}

\begin{Notation/Lemma}
Let $1$ denote $I_8\in G$,
$g\in G$ and $u\in U$,
and set
$G(g-1)G:= \{x(g-1)y \mid x,y\in G\}\subseteq V_G$,
$C_g^G:= \{1+x(g-1)y \mid x,y\in G\}=1+G(g-1)G \subseteq G$,
and
$C_u^U:= \{1+x(u-1)y \mid x,y\in G\}\cap U\subseteq C_u^G$.
\end{Notation/Lemma}

\begin{Lemma}\label{3D4,biorbit-G}
Let $1$ denote $I_8\in G$
and $g,h\in G$.
Then the following statements are equivalent:
\begin{enumerate}
\item[(1)] there exist $x,y\in G$ such that $g-1=x(h-1)y$,
\item[(2)] $C_g^G=C_h^G$, and
\item[(3)] $g\in C_h^G$.
\end{enumerate}
\end{Lemma}

\begin{Corollary}
 The set $\{C_g^G \mid g\in G\}$ forms a partition of $G$
 with respect to the equivalence relations of \ref{3D4,biorbit-G}.
If $g\in G$, then $C_g^G$ is a union of conjugacy classes of $G$.
\end{Corollary}

\begin{Lemma}\label{3D4,biorbit-U}
 Let $u,v\in U$.
 Then the following statements are equivalent:
\begin{enumerate}
\item[(1)] there exist $x,y\in G$ such that $u-1=x(v-1)y$,
\item[(2)] $C_u^U=C_v^U$, and
\item[(3)] $u\in C_v^U$.
\end{enumerate}
\end{Lemma}
\begin{Corollary}
 The set $\{C_u^U \mid u\in U\}$ forms a partition of $U$
 with respect to the equivalence relations of \ref{3D4,biorbit-U}.
 If $u\in U$, then $C_u^U$ is a union of conjugacy classes of $U$.
\end{Corollary}
We obtain a partition of ${^3{D}}_4^{syl}(q^3)$
by straightforward calculation.
\begin{Proposition}[A partition of ${^3{D}}_4^{syl}(q^3)$]\label{partition-3D4}
Let $T^{t_1^*}$ $({t_1^*}\in \mathbb{F}_{q^3})$ be the transversal for $t_1^*\mathbb{F}_{q}^+$ in $\mathbb{F}_{q^3}^+$.
Then the $C_u^U$ with $u\in U$ are given in Table \ref{table:partition-3D4}:
\begin{table}[!htp]
\caption{A partition of ${^3{D}}_4^{syl}(q^3)$}
\label{table:partition-3D4}
\begin{align*}
\begin{array}{|c|c|l|c|}\hline
u \in U
& \# C_u^U
&
\multicolumn{1}{c|}{C_u^U}
& |C_u^U|\\
\hline
\hline
I_8 & 1 & x(0,0,0,0,0,0) & 1\\\hline
x_6(t_6^*)
& q-1 & x(0,0,0,0,0,t_6^*) & 1\\\hline
x_5(t_5^*) & q-1
& x(0,0,0,0,t_5^*,s_6),{\ }
\forall{\,}
   s_6\in \mathbb{F}_q
& q\\\hline
x_4(t_4^*) & q^3-1
& x(0,0,0,t_4^*,s_5,s_6),{\ }
\forall{\,}
s_5,s_6\in \mathbb{F}_q
& q^2\\\hline
x_3(t_3^*) & q^3-1
& x(0,0,t_3^*,s_4,s_5,s_6),{\ }
\forall{\,}
s_4\in \mathbb{F}_{q^3}, s_5,s_6\in \mathbb{F}_q
& q^5\\\hline
\hline
x_2(t_2^*)x_4(t_4^*) & (q-1)(q^3-1)
& x(0,t_2^*,s_3,t_4^*-\frac{s_3^{q+1}}{t_2^*},s_5,s_6),{\ }
\forall{\,} s_3 \in \mathbb{F}_{q^3}, s_5,s_6\in \mathbb{F}_q
&  q^5 \\\hline
x_2(t_2^*)x_5(t_5)  & (q-1)q
&
x(0,t_2^*,s_3,-\frac{s_3^{q+1}}{t_2^*},t_5+\frac{s_3^{q^2+q+1}}{{t_2^*}^2},s_6),{\ }
\forall{\,} s_3 \in \mathbb{F}_{q^3}, s_6\in \mathbb{F}_q
& q^4\\\hline
\hline
x_1(t_1^*)x_3({\bar{t}}_3^{t_1^*}) & (q^3-1)q^2
& x(t_1^*,0,\bar{t}_3^{t_1^*}+t_1^*s,s_4,s_5,s_6),{\ }
\forall{\,}
s_4 \in \mathbb{F}_{q^3}, s,s_5,s_6\in \mathbb{F}_q
& q^6 \\\hline
x_2(t_2^*)x_1(t_1^*) & (q-1)(q^3-1)
& x(t_1^*,t_2^*,s_3,s_4,s_5,s_6),{\ }
\forall{\,}
s_3,s_4 \in \mathbb{F}_{q^3}, s_5,s_6\in \mathbb{F}_q
& q^8\\\hline
\end{array}
\end{align*}
where ${\bar{t}}_3^{t_1^*}\in T^{t_1^*}$.
\end{table}
\end{Proposition}

\begin{Notation/Lemma}\label{superclass:ci ti-3D4}
Set
 \begin{align*}
  & C_6(t_6^*):= C_{x_6(t_6^*)}^U,\quad
  C_5(t_5^*):= C_{x_5(t_5^*)}^U,\quad
  C_4(t_4^*):= C_{x_4(t_4^*)}^U,\quad
  C_3(t_3^*):= C_{x_3(t_3^*)}^U,\\
  & C_2(t_2^*):= \big(\bigcup_{t_4^*\in \mathbb{F}_{q^3}^*}^{.}{ C_{x_2(t_2^*)x_4(t_4^*)}^U}\big)
                \dot{\bigcup}
                \big(\bigcup_{t_5\in \mathbb{F}_q}^{.}{ C_{x_2(t_2^*)x_5(t_5)}^U}\big),\\
  & C_{1,3}(t_1^*,{\bar{t}}_3^{t_1^*}):=
                 { C_{x_1(t_1^*)x_3({\bar{t}}_3^{t_1^*})}^U},\quad
  C_{1,2}(t_1^*,t_2^*):= C_{x_2(t_2^*)x_1(t_1^*)}^U,\quad
  C_0: = \{1_U\}=\{ I_8 \}.
 \end{align*}
Note that these sets form a partition of $U$, denoted by $\mathcal{K}$.
\end{Notation/Lemma}

\begin{Comparison}[Superclasses]
\label{com:superclasses-3D4}
\begin{itemize}
 \setlength\itemsep{0em}
\item [(1)]
The superclasses of $A_n(q)$
are the sets
$C_g^{AY}:=\{I_{n}+x(g-I_n)y\mid x,y \in A_n(q)\}$
for all $g\in A_n(q)$,
which are called \textbf{Andr\'{e}-Yan superclasses}
(see \cite[2.2.30 and 3.5.3]{Markus1}).
\item [(2)]
The superclasses of $D_n^{syl}(q)$ are the sets
$C_u^{AN}:=D_n^{syl}(q) \cap \{I_{2n}+x(u-I_{2n})y\mid x,y \in A_{2n}(q)\}$
for all $u\in D_n^{syl}(q)$
(see \cite[page 1279]{an2}),
which are called \textbf{Andr\'{e}-Neto superclasses}
(see \cite[3.5.5]{Markus1}).
\item [(3)]
The superclasses of ${^3}D_4^{syl}(q^3)$ are the elements of
$\mathcal{K}$
(see \ref{superclass:ci ti-3D4} and \ref{supercharacter theory-3D4}).
\end{itemize}
\end{Comparison}


\section{A supercharacter theory for ${{^3D}_4^{syl}}(q^3)$}
\label{sec: supercharacter theories-3D4}

In this section, we determine a supercharacter theory for ${{^3D}_4^{syl}}(q^3)$ (\ref{supercharacter theory-3D4}),
and establish the supercharacter table of ${{^3D}_4^{syl}}(q^3)$ in
Table \ref{table:supercharacter table-3D4}.
Let $U:={^3}D_4^{syl}(q^3)$,
$A_{12}^*,A_{13}^*,A_{15}^*\in \mathbb{F}_{q^3}^*$ and
$A_{16}^*,A_{17}^*,A_{23}^*\in \mathbb{F}_q^*$.
\begin{Definition}
\label{pre-supercharacter theory}
Let $G$ be a finite group.
Suppose that $\mathcal{K}$ is a partition of $G$
and that $\mathcal{X}$ is a set of (nonzero) complex characters of $G$,
such that
\begin{itemize}
 \setlength\itemsep{0em}
 \item [(a)] $|\mathcal{X}|=|\mathcal{K}|$,
 \item [(b)] every character $\chi \in \mathcal{X}$ is constant on each member of $\mathcal{K}$ and
 \item [(c)] the elements of $\mathcal{X}$ are pairwise orthogonal.
\end{itemize}
Then $(\mathcal{X},\mathcal{K})$ is called a
\textbf{pre-supercharacter theory}
for $G$.
The function $\varphi\colon G\to \mathbb{C}$ is called a
\textbf{superclass function},
if it satisfies (b).
In particular, the superclass functions form a $\mathbb{C}$-vector space.
\end{Definition}

\begin{Definition/Lemma}[\S 2 of \cite{di}/3.6.2 of \cite{Markus1}]
\label{supercharacter theory}

Let $G$ be a finite group
and $(\mathcal{X},\mathcal{K})$ be a {pre-supercharacter theory} for $G$.
For every $\chi \in \mathcal{X}$,
let $\mathrm{Irr}(\chi)$ denote the set of all irreducible constituents of $\chi$.
Set $\sigma_{\chi}:=\sum_{\psi \in \mathrm{Irr}(\chi)}{\psi(1)\psi}$.
Then the following statements are equivalent.
\begin{itemize}
 \setlength\itemsep{0em}
 \item [(1)] The set $\{1\}$ is a member of $\mathcal{K}$.
 \item [(2)] $\cup_{\chi\in \mathcal{X}}{\mathrm{Irr}(\chi)}=\mathrm{Irr}(G)$
              and every character $\chi\in \mathcal{X}$ is a constant multiple of $\sigma_{\chi}$.
 \item [(3)] Every irreducible character $\psi$ of $G$ is a constituent of one character
              $\chi\in \mathcal{X}$.
\end{itemize}
Then $(\mathcal{X},\mathcal{K})$
is called a \textbf{supercharacter theory} for $G$,
if one of the three statements holds.
We refer to the elements of $\mathcal{X}$ as \textbf{supercharacters},
and to the elements of $\mathcal{K}$ as \textbf{superclasses}
of $G$.
A $\mathbb{C}G$-module is called a $\mathbb{C}G$-\textbf{supermodule},
if it affords a supercharacter of $G$.
Note that $\sum_{\chi \in \mathcal{X}}{\sigma_{\chi}}=\mathrm{reg}_G$, the regular character of $G$.
\end{Definition/Lemma}

\begin{Notation/Lemma}\label{notation:supermodules-3D4}
Let $A=(A_{ij})\in V$,
and set
 \begin{align*}
 M{(A_{12}e_{12}+A_{23}e_{23})}:=&
 \mathbb{C}\mathcal{O}_{U}([A_{12}e_{12}+A_{23}e_{23}])
=\mathbb{C}[A_{12}e_{12}+A_{23}e_{23}],
\end{align*}
\begin{align*}
 M{(A_{13}^*e_{13}+{\bar{A}_{12}^{A_{13}^*}} e_{12})}:=&
{\mathbb{C}\left\{
\left[{\newcommand{\mc}[3]{\multicolumn{#1}{#2}{#3}}
\begin{array}{cccccccc}\cline{2-7}
\mc{1}{c|}{ × }
& \mc{1}{c|}{
\rule{0pt}{15pt}
\bar{A}_{12}^{A_{13}^*}+s_2A_{13}^*}
& \mc{1}{c|}{{A_{13}^*}}
& \mc{1}{c|}{×}
& \mc{1}{c|}{×} & \mc{1}{c|}{×} & \mc{1}{c|}{×} &  × \\\cline{2-7}
× & \mc{1}{c|}{ × } & \mc{1}{c|}{} & × & × & × &  ×  & ×\\\cline{3-3}
\end{array}
}\right]
{\, }\middle|{\, } s_2\in \mathbb{F}_{q}
\right\}
}\\
=& \mathbb{C}\mathcal{O}_{U}([A_{13}^*e_{13}+{\bar{A}_{12}^{A_{13}^*}} e_{12}]),
\end{align*}
\begin{align*}
M{({A_{15}^*}^qe_{14}+A_{15}^*e_{15})}:=
&
{\mathbb{C}\left\{
\left[{\newcommand{\mc}[3]{\multicolumn{#1}{#2}{#3}}
\begin{array}{cccccccc}\cline{2-7}
\mc{1}{c|}{ × }
& \mc{1}{c|}{A_{12}}
& \mc{1}{c|}{A_{13}}
& \mc{1}{c|}{{{A_{15}^*}^q}}
& \mc{1}{c|}{{A_{15}^*}} & \mc{1}{c|}{×} & \mc{1}{c|}{×} &  × \\\cline{2-7}
× & \mc{1}{c|}{ × } & \mc{1}{c|}{A_{23}} & × & × & × &  ×  & ×\\\cline{3-3}
\end{array}}
\right]
{\, }\middle|{\, }
{\left\{\begin{array}{l}
A_{12},A_{13}\in \mathbb{F}_{q^3}\\
A_{23}\in \mathbb{F}_{q}
\end{array}\right.}
\right\}
}\\
=& \bigoplus_{A_{23}\in \mathbb{F}_{q}}
\mathbb{C}\mathcal{O}_{U}
([{A_{15}^*}^qe_{14}+A_{15}^*e_{15}+A_{23}e_{23}]),
\end{align*}
\begin{align*}
 M{(A_{16}^*e_{16})}:=
&
{\mathbb{C}\left\{
\left[{\newcommand{\mc}[3]{\multicolumn{#1}{#2}{#3}}
\begin{array}{cccccccc}\cline{2-7}
\mc{1}{c|}{ × }
& \mc{1}{c|}{A_{12}}
& \mc{1}{c|}{A_{13}}
& \mc{1}{c|}{A_{15}^q}
& \mc{1}{c|}{A_{15}} & \mc{1}{c|}{{A_{16}^*}} & \mc{1}{c|}{×} &  × \\\cline{2-7}
× & \mc{1}{c|}{ × } & \mc{1}{c|}{A_{23}} & × & × & × &  ×  & ×\\\cline{3-3}
\end{array}
}\right]
{\, }\middle|{\, }
{\left\{\begin{array}{l}
A_{12},A_{13},A_{15}\in \mathbb{F}_{q^3}\\
A_{23}\in \mathbb{F}_{q}
\end{array}\right.}
\right\}
}\\
=& \bigoplus_{\substack{A_{13}\in \mathbb{F}_{q^3}\\A_{23}\in \mathbb{F}_{q}}}
\mathbb{C}\mathcal{O}_{U}
([A_{16}^*e_{16}+A_{13}e_{13}+A_{23}e_{23}]),
\end{align*}
\begin{align*}
 M{(A_{17}^*e_{17})}:=
&
{\mathbb{C}\left\{
\left[{\newcommand{\mc}[3]{\multicolumn{#1}{#2}{#3}}
\begin{array}{cccccccc}\cline{2-7}
\mc{1}{c|}{ × }
& \mc{1}{c|}{A_{12}}
& \mc{1}{c|}{A_{13}}
& \mc{1}{c|}{A_{15}^q}
& \mc{1}{c|}{A_{15}} & \mc{1}{c|}{A_{16}} & \mc{1}{c|}{{A_{17}^*}} &  × \\\cline{2-7}
× & \mc{1}{c|}{ × } & \mc{1}{c|}{} & × & × & × &  ×  & ×\\\cline{3-3}
\end{array}
}\right]
{\, }\middle|{\, }
{\left\{\begin{array}{l}
A_{12},A_{13},A_{15}\in \mathbb{F}_{q^3}\\
A_{16}\in \mathbb{F}_{q}
\end{array}\right.}
\right\}
}\\
=& \bigoplus_{A_{12}\in \mathbb{F}_{q^3}}
\mathbb{C}\mathcal{O}_{U}
([A_{17}^*e_{17}+A_{12}e_{12}]).
\end{align*}
 Denote by $\mathcal{M}$
 the set of all of the above $\mathbb{C}U$-modules.
\end{Notation/Lemma}

\begin{Lemma}\label{notation:supermodules and G-module-3D4}
Let $A=(A_{ij})\in V$ and $G:=G_8(q^3)$.
Then
all $G$-orbit modules are irreducible,
and every $U$-module in $\mathcal{M}$ is
a direct sum of restrictions of some $G_8(q^3)$-orbit modules to ${{^3D}_4^{syl}}(q^3)$
as follows:
 \begin{align*}
 & M{(A_{12}e_{12}+A_{23}e_{23})}=
  \mathrm{Res}^G_U \mathbb{C}\mathcal{O}_{G}([A_{12}e_{12}+A_{23}e_{23}]),\\
& M{(A_{13}^*e_{13}+{\bar{A}_{12}^{A_{13}^*}} e_{12})}=
 \mathrm{Res}^G_U \mathbb{C}\mathcal{O}_{G}([A_{13}^*e_{13}+{\bar{A}_{12}^{A_{13}^*}} e_{12}]),\\
& M{({A_{15}^*}^qe_{14}+A_{15}^*e_{15})}=
  \bigoplus_{A_{23}\in \mathbb{F}_{q}}
\mathrm{Res}^G_U \mathbb{C}\mathcal{O}_{G}
([{A_{15}^*}^qe_{14}+A_{15}^*e_{15}+A_{23}e_{23}]),\\
&  M{(A_{16}^*e_{16})}= \bigoplus_{A_{23}\in \mathbb{F}_{q}}
\mathrm{Res}^G_U \mathbb{C}\mathcal{O}_{G}
([A_{16}^*e_{16}+A_{23}e_{23}]),
 \quad
 M{(A_{17}^*e_{17})}=
 \mathrm{Res}^G_U \mathbb{C}\mathcal{O}_{G}
([A_{17}^*e_{17}]).
\end{align*}
\end{Lemma}

\begin{Notation}\label{set of supercharacters-3D4}
 Let $M\in \mathcal{M}$,
 and let the complex character
of the $\mathbb{C}U$-module $M$ be denoted by $\Psi_M$.
We set
 $ \mathcal{X}:=\{\Psi_M \mid M\in \mathcal{M}\}$.
\end{Notation}

\begin{Corollary}
Let $A=(A_{ij})\in V$, and $\psi_A$ be the character of $\mathbb{C}\mathcal{O}_U([A])$.
Then
 \begin{alignat*}{2}
 &\Psi_{M(A_{12}e_{12}+A_{23}e_{23})}= \psi_{A_{12}e_{12}+A_{23}e_{23}},\\
 &\Psi_{M(A_{13}^*e_{13}+\bar{A}_{12}^{A_{13}^*} e_{12})}
                        = {\psi_{\bar{A}_{12}^{A_{13}^*}e_{12}+A_{13}^*e_{13}}},
& \quad
 &\Psi_{M({A_{15}^*}^qe_{14}+A_{15}^*e_{15})}
                         = \sum_{A_{23}\in \mathbb{F}_q}
                           {\psi_{A_{23}e_{23}+{A_{15}^*}^qe_{14}+A_{15}^*e_{15}}},\\
 &\Psi_{M(A_{16}^*e_{16})}= \sum_{\substack{A_{13}\in \mathbb{F}_{q^3}\\A_{23}\in \mathbb{F}_q}}
                           {\psi_{A_{13}e_{13}+A_{23}e_{23}+A_{16}^*e_{16}}},
&\quad
 &\Psi_{M(A_{17}^*e_{17})}= \sum_{A_{12}\in \mathbb{F}_{q^3}}
                           {\psi_{A_{12}e_{12}+A_{17}^*e_{17}}}.
 \end{alignat*}
\end{Corollary}

\begin{Proposition}[Supercharacter theory for ${{^3D}_4^{syl}}(q^3)$]\label{supercharacter theory-3D4}
$(\mathcal{X},\mathcal{K})$ is a supercharacter theory for ${{^3D}_4^{syl}}(q^3)$,
where $\mathcal{K}$ is defined in \ref{superclass:ci ti-3D4},
and $\mathcal{X}$ is defined in \ref{set of supercharacters-3D4}.
\end{Proposition}
\begin{proof}
 By \ref{superclass:ci ti-3D4}, $\mathcal{K}$ is a partition of $U$.
 We know  $\mathcal{X}$ is a set of nonzero complex characters of $U$.
\begin{itemize}
\setlength\itemsep{0em}
 \item [(a)] {\it{Claim that $|\mathcal{X}|=|\mathcal{K}|$}}.
By \ref{superclass:ci ti-3D4}, \ref{notation:supermodules-3D4} and \ref{set of supercharacters-3D4},
$|\{\Psi_{M{(A_{17}^*e_{17})}} \mid A_{17}^*\in \mathbb{F}_q^* \}|
=|\{{M{(A_{17}^*e_{17})}} \mid  A_{17}^*\in \mathbb{F}_q^*\}|
=|\{C_6(t_6^*) \mid t_6^* \in \mathbb{F}_q^* \}|$.
Similarly, we obtain $|\mathcal{X}|=|\mathcal{K}|$.
 \item [(b)] {\it Claim that the characters $\chi \in \mathcal{X}$ are constant on
          the members of $\mathcal{K}$}.
 Let $A \in \mathfrak{F}_4$ and
 \begin{align*}
 \mathcal{B}_{15}(A_{15}^*):=
{ \left\{
 {\newcommand{\mc}[3]{\multicolumn{#1}{#2}{#3}}
\begin{array}{cccccccc}\cline{2-7}
\mc{1}{c|}{ × }
& \mc{1}{c|}{C_{12}}
& \mc{1}{c|}{C_{13}}
& \mc{1}{c|}{{{A_{15}^*}^q}}
& \mc{1}{c|}{{A_{15}^*}} & \mc{1}{c|}{×} & \mc{1}{c|}{×} &  × \\\cline{2-7}
× & \mc{1}{c|}{ × } & \mc{1}{c|}{C_{23}} & × & × & × &  ×  & ×\\\cline{3-3}
\end{array}}
{\, }\middle|{\, }
{\left\{\begin{array}{l}
C_{12},C_{13}\in \mathbb{F}_{q^3}\\
C_{23}\in \mathbb{F}_{q}
\end{array}\right.}
\right\}
}.
\end{align*}
Let $y\in U$.
Then we have that
 \begin{align*}
 \Psi_{M({A_{15}^*}^qe_{14}+A_{15}^*e_{15})}(y)=
     \sum_{\substack{C\in \mathcal{B}_{15}(A_{15}^*)\\C.y=C}}{\chi_C(y)}
 =\sum_{\substack{C\in \mathcal{B}_{15}(A_{15}^*)\\ y\in \mathrm{Stab}_U(C)}}{\chi_C(y)}.
 \end{align*}
 If $y=x(0,0,0,t_4,t_5,t_6)\in C_0\cup C_4(t_4^*)\cup C_5(t_5^*)\cup C_6(t_6^*) \subseteq \mathcal{K}$,
then
 $y\in \mathrm{Stab}_U(C)$ for all $C\in \mathcal{B}_{15}(A_{15}^*)$ by \ref{prop: 3D4-stab}.
 Thus
\begin{align*}
\Psi_{M({A_{15}^*}^qe_{14}+A_{15}^*e_{15})}(y)=&
               \sum_{C\in \mathcal{B}_{15}(A_{15}^*)}{\chi_C(y)}
               =q^7\cdot
               {\vartheta\pi_q({A_{15}^*}^qt_4+{A_{15}^*}t_4^{q^2})}.
\end{align*}
If $y\in C_{1,2}(t_1^*,t_2^*)\cup C_{1,3}(t_1^*,{\bar{t}}_3^{t_1^*})\cup C_3(t_3^*) \subseteq \mathcal{K}$,
then
 $y \notin \mathrm{Stab}_U(C)$ for all $C\in \mathcal{B}_{15}(A_{15}^*)$ by \ref{prop: 3D4-stab}.
 Thus
$
\Psi_{M({A_{15}^*}^qe_{14}+A_{15}^*e_{15})}(y)=0$.

If $y=x(0,t_2^*,s_3,s_4,s_5,s_6)\in C_2(t_2^*) \subseteq \mathcal{K}$,
then by \ref{prop: 3D4-stab} we have that
\begin{align*}
&   \Psi_{M({A_{15}^*}^qe_{14}+A_{15}^*e_{15})}(y)
=  \sum_{\substack{C\in \mathcal{B}_{15}(A_{15}^*)
              \\{A_{15}^*}^qs_3^q+A_{15}^*s_3^{q^2}+C_{13}t_2^*=0}}{\chi_C(y)}\\
= & \sum_{\substack{C_{12}\in \mathbb{F}_{q^3}\\C_{23}\in \mathbb{F}_{q}}}{
   \vartheta \kappa_q\left({\,}
{%
\newcommand{\mc}[3]{\multicolumn{#1}{#2}{#3}}
\begin{array}{c|c|cccc}\hline
\mc{1}{|c|}{C_{12}}
& -\frac{{A_{15}^*}^qs_3^{q}+A_{15}^*s_3^{q^2}}{t_2^*}
& \mc{1}{c|}{{A_{15}^*}^q} & \mc{1}{c|}{{A_{15}^*}} & \mc{1}{c|}{×} & \mc{1}{c|}{×}\\\hline
× & C_{23} & × & × & × & ×\\\cline{2-2}
\end{array}
}%
,{\,}
{%
\newcommand{\mc}[3]{\multicolumn{#1}{#2}{#3}}
\begin{array}{c|c|cccc}\hline
\mc{1}{|c|}{0} & -s_3 & \mc{1}{c|}{s_4}
& \mc{1}{c|}{s_4^{q^2}} & \mc{1}{c|}{*} & \mc{1}{c|}{*}\\\hline
× & t_2^* & × & × & × & ×\\\cline{2-2}
\end{array}
}%
{\,}\right)}\\
= & q^3\cdot {\vartheta\pi_q\Big(\frac{{A_{15}^*}^qs_3^{q+1}+A_{15}^*s_3^{q^2+1}}{t_2^*}
                                  +{A_{15}^*}^qs_4+{A_{15}^*}s_4^{q^2}\Big)}
        \cdot \sum_{C_{23}\in \mathbb{F}_{q}}{
                         \vartheta\pi_q(C_{23}t_2^*)}
= 0.
\end{align*}
Similarly, we calculate the other values of the Table \ref{table:supercharacter table-3D4}.
Thus the claim is proved.
 \item [(c)] The elements of $\mathcal{X}$ are pairwise orthogonal by \ref{3D4-orth.}.
 \item [(d)] The set $\{I_8\}$ is a member of $\mathcal{K}$.
\end{itemize}
Therefore,
 $(\mathcal{X},\mathcal{K})$ is a supercharacter theory for ${{^3D}_4^{syl}}(q^3)$.
\end{proof}
\begin{sidewaystable}
\caption{Supercharacter table of ${{^3D}_4^{syl}}(q^3)$}
\label{table:supercharacter table-3D4}
{
\footnotesize
\begin{align*}
\renewcommand\arraystretch{1.5}
\begin{array}{l|cccccccc}
×
& C_0
& C_1(t_1^*, {\bar{t}}_3^{t_1^*})
& C_2(t_2^*)
& C_{1,2}(t_1^*,t_2^*)
& C_3(t_3^*)
& C_4(t_4^*)
& C_5(t_5^*)
& C_6(t_6^*)\\
\hline
\Psi_{M{(0)}}
& 1 & 1 & 1 & 1 & 1 & 1 & 1 & 1\\
\Psi_{M{(A_{12}^*e_{12})}}
& 1
& \vartheta \pi_q (A_{12}^*t_1^*)
& 1
& \vartheta \pi_q (A_{12}^*t_1^*)
& 1 & 1 & 1 & 1 \\
\Psi_{M{(A_{23}^*e_{23})}}
& 1
& 1
& \vartheta (A_{23}^*t_2^*)
& \vartheta (A_{23}^*t_2^*)
& 1 & 1 & 1 & 1 \\
\Psi_{M{(A_{12}^*e_{12}+A_{23}^*e_{23})}}
& 1
& \vartheta \pi_q (A_{12}^*t_1^*)
& \vartheta (A_{23}^*t_2^*)
& \begin{array}{l}
\vartheta \pi_q (A_{12}^*t_1^*)\\
\cdot \vartheta (A_{23}^*t_2^*)
\end{array}
& 1 & 1 & 1 & 1
\\
\Psi_{M{(A_{13}^*e_{13}+\bar{A}_{12}^{A_{13}^*} e_{12})}}
& q
& \begin{smallmatrix}
   \vartheta \pi_q(\bar{A}_{12}^{A_{13}^*}t_1^*-\bar{t}_3^{t_1^*}A_{13}^*)\\
   \cdot \sum_{r_2\in \mathbb{F}_{q}}\vartheta \pi_q( -A_{13}^*t_1^*r_2)
  \end{smallmatrix}
& 0 & 0
& \begin{array}{l}
\vartheta\pi_q(-A_{13}^*t_{3}^*)\\
  \cdot q
  \end{array}
& q & q & q
\\
\Psi_{M{({A_{15}^*}^qe_{14}+A_{15}^*e_{15})}}
& q^7
& 0 & 0 & 0
& 0
& \begin{array}{l}
\vartheta\pi_q({A_{15}^*}^qt_4^*)\\
\cdot \vartheta\pi_q({A_{15}^*}{t_4^*}^{q^2})\\
\cdot  q^7
  \end{array}
& q^7 & q^7
\\
 \Psi_{M{(A_{16}^*e_{16})}}
& q^{10}
& 0 & 0 & 0
& 0
& 0
& \begin{array}{l}
  \vartheta(A_{16}^*t_{5}^*)\\
  \cdot q^{10}
  \end{array}
& q^{10}
\\
\Psi_{M{(A_{17}^*e_{17})}}
& q^{10}
& 0 & 0 & 0
& 0 & 0 & 0
& \begin{array}{l}
  \vartheta(A_{17}^*t_{6}^*)\\
  \cdot q^{10}
  \end{array}
\end{array}
\end{align*}
}
\end{sidewaystable}

\begin{Corollary}
By \ref{supercharacter theory-3D4},
the modules of $\mathcal{M}$ (see \ref{notation:supermodules-3D4}) are $\mathbb{C}G$-supermodules,
and
the number of the supercharacters of ${^3{D}_4^{syl}(q^3)}$ is
\begin{align*}
 & |\mathcal{X}|=|\mathcal{M}|=|\mathcal{K}|
= q^5+q^4+q^3-q^2+2q-3\\
=& (q-1)^5+6(q-1)^4+15(q-1)^3+18(q-1)^2+12(q-1)+1.
\end{align*}
\end{Corollary}

In the following proposition, we use the notations of \cite{sun2016arxiv}.
\begin{Proposition}[Supercharacters and irreducible characters]\label{relation: superchar. and irr.-3D4}
The following relations between supercharacters and irreducible characters
of ${{^3D}_4^{syl}}(q^3)$ are obtained.
\begin{flalign*}
&\Psi_{M{(A_{17}^*e_{17})}}= q^3 \sum_{A_{12} \in \mathbb{F}_{q^3}}{\chi_{6,q^4}^{A_{17}^*,A_{12}}},
& \qquad
&\Psi_{M{(A_{16}^*e_{16})}}= q^3 \sum_{\substack{A_{23}\in \mathbb{F}_{q}\\A_{13} \in \mathbb{F}_{q^3}}}
                             {\chi_{5,q^3}^{A_{16}^*,A_{23},A_{13}}},\\
&\Psi_{M{({A_{15}^*}^qe_{14}+A_{15}^*e_{15})}}= q^3 \sum_{A_{23}\in \mathbb{F}_{q}}
                                               {\chi_{4,q^3}^{A_{15}^*,A_{23}}},
& \qquad
&\Psi_{M{(A_{13}^*e_{13}+\bar{A}_{12}^{A_{13}^*} e_{12})}}
                            =\chi_{3,q}^{A_{13}^*,\bar{A}_{12}^{A_{13}^*}},\\
&\Psi_{M{(A_{12}e_{12}+A_{23}e_{23})}}= \chi_{lin}^{A_{12},A_{23}}.
\end{flalign*}
\end{Proposition}
\begin{proof}
Take
$\Psi_{M{({A_{15}^*}^qe_{14}+A_{15}^*e_{15})}}= q^3 \sum_{A_{23}\in \mathbb{F}_{q}}
                                               {\chi_{4,q^3}^{A_{15}^*,A_{23}}}$
for example.
We must show that
\begin{align*}
\Psi_{M{({A_{15}^*}^qe_{14}+A_{15}^*e_{15})}}(u)= q^3 \sum_{A_{23}\in \mathbb{F}_{q}}
                                               {\chi_{4,q^3}^{A_{15}^*,A_{23}}(u)}
\quad \text{for all } u\in U.
\end{align*}
By  Table \ref{table:supercharacter table-3D4}
and the character table of ${{^3D}_4^{syl}}(q^3)$ (see \cite{sun2016arxiv}),
it is sufficient to prove that
\begin{align*}
0=\Psi_{M{({A_{15}^*}^qe_{14}+A_{15}^*e_{15})}}(u)= q^3 \sum_{A_{23}\in \mathbb{F}_{q}}
                                               {\chi_{4,q^3}^{A_{15}^*,A_{23}}(u)}
\qquad \text{for all } u\in C_2(t_2^*)\subseteq U.
\end{align*}
Let $u:=x(0,t_2^*,t_3,t_4,t_5,t_6)$.
Then
\begin{align*}
 & q^3 \sum_{A_{23}\in \mathbb{F}_{q}}{\chi_{4,q^3}^{A_{15}^*,A_{23}}(u)}\\
=& q^3 \sum_{\substack{A_{23}\in \mathbb{F}_{q}\\r_1\in \mathbb{F}_{q^3}}}
             \vartheta \pi_q
    ( A_{23}t_2^*
    +{A_{15}^*}^q(t_4-t_2^*r_1^{q+1}+r_1t_3^q+r_1^qt_3)
        +A_{15}^*(t_4^{q^2}-t_2^*r_1^{q^2+1}+r_1^{q^2}t_3+r_1t_3^{q^2}) )\\
=& 0.
\end{align*}
Thus $\Psi_{M{({A_{15}^*}^qe_{14}+A_{15}^*e_{15})}}= q^3 \sum_{A_{23}\in \mathbb{F}_{q}}
                                               {\chi_{4,q^3}^{A_{15}^*,A_{23}}}$
is proved.
Similarly, we get the other formulae.
\end{proof}

\begin{Comparison}[Supercharacters]
\label{com:supercharacter theory-3D4}
\begin{itemize}
\setlength\itemsep{0em}
\item [(1)]
Supercharacters of $A_n(q)$
are the characters of staircase $A_n(q)$-modules,
which are called \textbf{Andr\'{e}-Yan supercharacters}
(see \cite[2.2.27]{Markus1}).
\item [(2)]
Supercharacters of $D_n^{syl}(q)$
are the characters of the sums of all hook-separated staircase $D_n^{syl}(q)$-modules with the same verge
(see \cite[3.6.13 and 3.6.16]{Markus1}
and \cite[page 1278]{an2}),
which are called \textbf{Andr\'{e}-Neto supercharacters}.
\item [(3)]
Supercharacters of ${^3}D_4^{syl}(q^3)$
of families $\mathfrak{F}_4$, $\mathfrak{F}_5$ and $\mathfrak{F}_6$
are the characters of the sums of all hook-separated staircase ${^3}D_4^{syl}(q^3)$-modules
with the same 1st verge
(see \ref{notation:supermodules-3D4} and \ref{supercharacter theory-3D4}).
\end{itemize}
\end{Comparison}


\section*{Acknowledgements}
This paper is part of my PhD thesis \cite{sunphd} at the University of Stuttgart, Germany,
so I am deeply grateful to my supervisor Richard Dipper.
I also would like to thank  Markus Jedlitschky,
Mathias Werth
and Yichuan Yang for the
helpful discussions and valuable suggestions.


\bibliographystyle{abbrv}
\bibliography{bibliography3D4supercharacter}

\end{document}